\documentclass[hidelinks,onefignum,onetabnum,final]{siamart220329}

\usepackage{lipsum}
\usepackage{amsfonts}
\usepackage{graphicx}
\usepackage{epstopdf}
\usepackage{algorithmic}
\usepackage{enumitem}
\usepackage{amsmath} 
\usepackage{amssymb}
\usepackage{mathtools}
\usepackage{tikz}
\usetikzlibrary{shapes.geometric, arrows}
\tikzstyle{startstop} = [rectangle, rounded corners, minimum width=3cm, minimum height=1cm, text centered, draw=black, fill=red!30]
\tikzstyle{process} = [rectangle, minimum width=3.5cm, minimum height=1cm, text centered, draw=black, fill=orange!30]
\tikzstyle{decision} = [diamond, minimum width=3cm, minimum height=1cm, text centered, draw=black, fill=green!30]
\tikzstyle{arrow} = [thick,->,>=stealth]
\tikzstyle{data} = [rectangle, minimum width=3.5cm, minimum height=1cm, text centered, draw=black, fill=blue!30]

\ifpdf
  \DeclareGraphicsExtensions{.eps,.pdf,.png,.jpg}
\else
  \DeclareGraphicsExtensions{.eps}
\fi

\newsiamremark{remark}{Remark}
\newsiamremark{hypothesis}{Hypothesis}
\newsiamremark{example}{Example}
\crefname{hypothesis}{Hypothesis}{Hypotheses}
\newsiamthm{claim}{Claim}

\title{Numerical analysis of the parallel orbital-updating approach for eigenvalue problems\thanks{Submitted to the editors DATE.
\funding{This work was supported by the National Key R \& D Program of China under grants 2019YFA0709600 and 2019YFA0709601, the National Natural Science Foundation of China under grant 12021001 and 92270206. B. Yang also acknowledges the support from the Fundamental Research Funds for the Central Universities and the disciplinary funding of Central University of Finance and Economics.}}}

\author{Xiaoying Dai\thanks{SKLMS, Academy of Mathematics and Systems Science, Chinese Academy of Sciences, Beijing 100190, China; and School of Mathematical Sciences, University of Chinese Academy of Sciences, Beijing 100049, China (\email{daixy@lsec.cc.ac.cn}, \email{liyan2021@lsec.cc.ac.cn}, \email{azhou@lsec.cc.ac.cn}).} \and
	Yan Li\footnotemark[2]
\and Bin Yang\thanks{School of Statistics and Mathematics, Central University of Finance and Economics, Beijing 102206, China (\email{binyang@lsec.cc.ac.cn}).}
\and Aihui Zhou\footnotemark[2]}

\usepackage{amsopn}

\newcommand{\inner}[1]{\left\langle #1 \right\rangle}
\usepackage[version=4]{mhchem}

\begin{document}

\maketitle

\begin{abstract}
The parallel orbital-updating approach is an orbital/eigenfunction iteration based approach for solving eigenvalue problems when many eigenpairs are required. It has been proven to be efficient, for instance, in electronic structure calculations. In this paper, based on the investigation of a quasi-orthogonality, we present the numerical analysis of the parallel orbital-updating approach for linear eigenvalue problems, including convergence and error estimates of the numerical approximations.
\end{abstract}



\section{Introduction}Eigenvalue problems are typical models in scientific and engineering computing.  For instance, Hartree–Fock type and Kohn-Sham equations are widely used mathematical models in electronic structure calculations. The eigenvalues and their corresponding eigenfunctions of these equations provide detailed information about the properties of atoms, molecules, and solids, helping to predict chemical reactions, material properties, and physical behaviors (see e.g. \cite{dai2011finite, kaxiras2003atomic, kohn1965self, martin2020electronic}). 

In electronic structure calculations of a large system, the approximations of many eigenpairs are required. With discretization and the self-consistent field iteration \cite{kresse1996efficient, martin2020electronic, payne1992iterative}, solving the Hartree–Fock type equations or the Kohn-Sham equations is then transformed into repeatedly solving some large scale algebraic eigenvalue problems. It is known that the computational cost of solving such large scale eigenvalue problems is high. In particular, the solving process often requires large scale orthogonalizing operations, which demand global summation operations and limit large scale parallelization.  Nowadays, the computational scale is limited for systems with hundreds to thousands of atoms. Since applications demand and supercomputers are available, it is significant to develop scalable and parallelizable numerical methods to solve such eigenvalue problems.

To reduce the computational cost and improve the parallel scalability, a so-called parallel orbital-updating (ParO) approach has been proposed in \cite{dai2014parallel} and developed in \cite{dai2021parallel, oliveira2020cecam, pan2017parallel} for solving eigenvalue problems and their equivalent models resulting from electronic structure calculations. We mention that there are also some other methods for approximating eigenpairs in the literature such as the density matrix based algorithm \cite{polizzi2009density}, the subspace iteration algorithm  \cite{saad2016analysis}, and the projection method based on the root-finding of the analytic function \cite{sakurai2003projection}. With ParO, we avoid solving the large scale eigenvalue problem directly and instead solve some independent large scale source problems and small scale projected eigenvalue problems to obtain approximate eigenpairs.  Moreover, we see from the numerical experiments in \cite{dai2014parallel, pan2017parallel} that the stiff matrix corresponding to the small scale eigenvalue problem is almost diagonal, 
which may further reduce the computation cost. Because of their independence, these source problems can be solved intrinsically in parallel. For each source problem, the standard parallel strategies can be applied. It then allows a two-level parallelization: one level of parallelization is obtained by partitioning these source problems into different groups of processors, and another level of parallelization is obtained by assigning each source problem to several processors contained in each group. This two-level parallelization demonstrates that ParO has great potential for large scale calculations. The numerical experiments in \cite{dai2014parallel, dai2021parallel, pan2017parallel} show that ParO is efficient, of good scalability of parallelization, and  can produce highly accurate approximations. We conclude that ParO is a powerful parallel computing approach for solving eigenvalue problems, in which many eigenpairs are required, and has been integrated into the electronic structure calculation software Quantum ESPRESSO \cite{quantum}. However, up to now, there is no mathematical justification for ParO.

The purpose of this paper is to present the numerical analysis of ParO for linear eigenvalue problems. We observe that during the implementation process of ParO, we are able to obtain approximately orthogonal eigenfunctions, which we call quasi-orthogonal eigenfunctions. ParO can be viewed as utilizing the quasi-orthogonal approximations, for which the computational cost is lower, to obtain orthogonal approximations. Our numerical analysis starts from the introduction and investigation of a quasi-orthogonality, which plays a crucial role in our
investigations on approximations of eigenvalue problems. We understand that the presence of both single eigenvalues and multiple eigenvalues renders traditional methods for analyzing single eigenvalues no longer applicable. The difficulty for the case of multiple eigenvalues lies in the fact that the traditional measure for the eigenfunction errors is not valid anymore, because the approximate eigenfunctions obtained in iterations may not approximate the same eigenfunction. Instead of focusing on particular eigenfunctions, in our analysis, we employ the eigenspaces and the gap from the eigenspaces to their approximations, which brings additional analyzing complexities and requires sophisticated functional analysis.

Some approaches for constructing source problems in ParO have been proposed in \cite{dai2014parallel}. As a practical example, the shifted-inverse based ParO algorithm applies the shifted-inverse approach to construct some source problems and solves a small scale eigenvalue problem in each iteration to update the shifts to speed up the convergence \cite{dai2014parallel, pan2017parallel}. To analyze the convergence of the algorithm, we first study its simplified version, which fixes the shifts and does not carry out the steps of solving small scale eigenvalue problems in iterations. Within the framework of ParO, we demonstrate the convergence of numerical solutions produced by the simplified algorithm, which does not require sufficiently accurate initial guesses. Based on the numerical analysis of the simplified version, we then present a more general and informative convergence result of the shifted-inverse based ParO algorithm than the classical results of the shifted-inverse approach for simple eigenvalues mentioned in, e.g., \cite{arbenz2012lecture,parlett1998symmetric}.  To improve the numerical stability, a modified version is proposed in \cite{pan2017parallel}, which augments the projected subspace by using the residuals. We also provide a brief outline of the proof for the convergence of this modified algorithm. 

The rest of this paper is organized as follows. We recall some existing results of a model problem and introduce the relevant notation in Section 2, and provide some elementary analysis for the quasi-orthogonality in Section 3. In Section 4, we show the convergence and error estimates of the numerical approximations produced by ParO and its practical versions. We give some concluding remarks in Section 5. Finally, we present the corresponding detailed proofs in Appendix A.

\section{Preliminaries}In this section, we recall some existing results for an eigenvalue problem (including its finite dimensional approximations) that will be used.

\subsubsection*{Eigenvalue problem} Suppose  $H$ is a real separable Hilbert space with inner product $\inner{\cdot, \cdot}$ and norm $\Vert \cdot\Vert=\sqrt{\inner{\cdot,\cdot}}$. Consider an eigenvalue problem: find $\lambda\in\mathbb{R}$ and $0\neq u\in H$ such that
\begin{align}\label{eq:weak_form_leq}
    a(u,v)=\lambda b(u,v),\quad \forall v\in H,
\end{align}
where $a(\cdot,\cdot)$ and $b(\cdot,\cdot)$ are two symmetric bilinear forms over $H\times H$. We assume that
\begin{align*}
    a(v,w)\leqslant C_{a}\Vert v\Vert\Vert w\Vert, \quad \forall v,w\in H,
\end{align*}
and 
\begin{align*}
    a(v,v)\geqslant c_{a}\Vert v\Vert^{2},\quad v\in H,
\end{align*}
with constants $C_{a}, c_{a}>0$. It follows that $a(\cdot, \cdot)$ is an inner product and the induced norm $\Vert v\Vert_{a}=\sqrt{a(v,v)}$ is equivalent to $\Vert\cdot\Vert$ on $H$. We assume that $b(\cdot,\cdot)$ is another inner product of $H$ and $\Vert\cdot\Vert_{b}\equiv \sqrt{b(\cdot,\cdot)}$ is compact with respect to $\Vert\cdot\Vert$. For convenience, we shall denote $\Vert\cdot\Vert_{a}$ as $\Vert\cdot\Vert$ in this paper.

It is known that \cref{eq:weak_form_leq} has a countable sequence of real eigenvalues $0<\lambda_{1}<\lambda_{2}<\cdots$ and $\lambda_{i}$ has the multiplicity $d_{i} (i=1,2,\ldots)$. The indices of $\lambda_{i}$ are $(i,1),\ldots,(i,d_{i})$, that is 
\begin{align*}
    \lambda_{i-1}<\lambda_{i}=\lambda_{i1}=\cdots=\lambda_{id_{i}}<\lambda_{i+1}, \quad i=1,2,\ldots,
\end{align*}
with $\lambda_{0}=0, d_{0}=0$. 

For $1\leqslant j\leqslant d_{i}$ and $1\leqslant s\leqslant d_{r}$, denote $(i,j)<(r,s)$ if $i<r$ or $i=r, j<s$. Let $M(\lambda_{i})$ be the eigenspace corresponding to $\lambda_{i}$ and $\{u_{ij}\}_{j=1}^{d_{i}}$ be the orthonormal basis of $M(\lambda_{i})$, that is, $M(\lambda_{i})=\operatorname{span}\{u_{i1}, \ldots, u_{id_{i}}\}$ for $i=1,2,\ldots$ with $b(u_{ij}, u_{kl})=\delta_{ik}\delta_{jl}$, where $\delta_{ik}$ and $\delta_{jl}$ are the Kronecker delta.

We consider to obtain the smallest $N$ clustered eigenvalues of \cref{eq:weak_form_leq} and their corresponding eigenfunctions, and assume that there exists $q\in\mathbb{N}_{+}$ such that $ \sum_{i=1}^{q}d_{i}=N.$

\subsubsection*{An example} A typical example of \cref{eq:weak_form_leq} is an eigenvalue problem of a partial differential operator over a bounded domain. Let $\Omega\subset\mathbb{R}^{d}(d\geqslant1)$ be a bounded domain. We shall use the standard notation for Sobolev spaces $H^{1}(\Omega)$ with associated norms (see, e.g.,  \cite{adams2003sobolev}). Let $H=H_{0}^{1}(\Omega)=\{v\in H^{1}(\Omega):v|_{\partial\Omega}=0\}$ and $(\cdot, \cdot)$ be the standard $L^{2}$ inner product. Consider the eigenvalue problem: find $\lambda\in\mathbb{R}$ and $u\in H_{0}^{1}(\Omega)$ with $\Vert u\Vert_{L^2(\Omega)}=1$ such that
\begin{align*}
   -\nabla\cdot(A\nabla u)+cu=\lambda u,
\end{align*}
where $A:\Omega\rightarrow\mathbb{R}^{d\times d}$ is piecewise Lipschitz and symmetric positive definite, and $0\leqslant c\in L^{\infty}(\Omega)$. Its associate weak form reads that: find $\lambda\in\mathbb{R}$ and $0\neq u\in H_{0}^{1}(\Omega)$ such that
\begin{align*}
    a(u,v)=\lambda b(u,v),\quad \forall v\in H_{0}^{1}(\Omega),
\end{align*}
where 
\begin{align*}
    a(u,v)=(A\nabla u, \nabla v)+(cu,v),\quad b(u,v)=(u,v).
\end{align*}
We see that $a(\cdot, \cdot)$ and $b(\cdot, \cdot)$ satisfy the assumptions above. 
\begin{remark}
We mention that the results obtained in this paper are also valid for a more general bilinear form $a(\cdot, \cdot)$ with
\begin{align*}
    \|w\|^2_{H^1_0(\Omega)}-C_{1}^{-1}\|w\|_{L^2(\Omega)}\leqslant C_{2}a(w,w),\quad\forall w\in H^1_0(\Omega)
\end{align*}
holding for some constant $C_{1}, C_{2}>0$ (see,  e.g., Remark 2.9 in \cite{dai2008convergence}).
\end{remark}

\subsubsection*{Distance from one subspace to another} To carry out the analysis, we apply the following distance from the nontrivial subspace $U\subset H$ to the subspace $V\subset H$ (\cite{chatelin2011spectral, kato2013perturbation, knyazev2006new}):
    \begin{align*}
        \operatorname{dist}(U, V):=\sup_{u\in U, \Vert u\Vert=1}\inf_{v\in V}\Vert u-v\Vert.
    \end{align*}
Consistently, for any $u,v\in H$, we define
\begin{align*}
    \operatorname{dist}(u,v):=\operatorname{dist}\left(\operatorname{span}\{u\}, \operatorname{span}\{v\}\right).
\end{align*}
We see that $\operatorname{dist}(u,v)$ is actually the sine of the angle between $u$ and $v$, and is independent of the norms of vectors. Note that $\operatorname{dist}(U, V)=1$ when $\operatorname{dim}(U)>\operatorname{dim}(V)$.
We also observe that for $U, V\subset H$ with $\operatorname{dim}(U)=\operatorname{dim}(V)<\infty$, there holds that
     \begin{align}\label{ine:dist_exchange}
         \operatorname{dist}(U, V)=\operatorname{dist}(V,U).
     \end{align}

Define $\mathcal{P}_{V}$ to be the orthogonal projection from $H$ onto $V\subset H$ with respect to the inner product $a(\cdot,\cdot)$. 

\subsubsection*{Finite dimensional approximation} Let $V^{h}$ be a finite dimensional subspace of $H$ with $\operatorname{dim}(V^{h})=N_{g}\geqslant N$. The standard finite dimensional discretization of \cref{eq:weak_form_leq} is defined as follows: find $\lambda^{h}\in\mathbb{R}$ and $0\neq u^{h}\in V^{h}$ such that
\begin{align}\label{eq:fd_weak_form_leq}
    a(u^{h},v)=\lambda^{h} b(u^{h},v),\quad \forall v\in V^{h}.
\end{align}

 For convenience, we assume that there exists $p\geqslant q$ such that $N_{g}=\sum_{i=1}^{p}d_{i}$. Then we may order the eigenvalues of \cref{eq:fd_weak_form_leq} as follows:
\begin{align*}
    0<\lambda^{h}_{11}\leqslant\cdots\leqslant\lambda^{h}_{1d_{1}}\leqslant\cdots\leqslant\lambda^{h}_{pd_{p}}.
\end{align*}
  The assumption is adopted to simplify the notation in our numerical analysis. In fact, our analysis results in this paper hold for all $N_{g}\geqslant N$, regardless of whether $N_{g}=\sum_{i=1}^{p}d_{i}$ is satisfied. Assume that the corresponding eigenfunctions $u^{h}_{ij}$ for $(i,j)\leqslant(p,d_{p})$ satisfy that $b(u^{h}_{ij}, u^{h}_{kl})=\delta_{ik}\delta_{jl}$. For $i=1,\ldots,p$, set $M_{h}(\lambda_{i})=\operatorname{span}\{u^{h}_{i1}, \ldots, u^{h}_{id_{i}}\}$.

We obtain from the minimum-maximum principle \cite{babuvska1989finite, chatelin2011spectral} that
\begin{align*}
    \lambda_{i}\leqslant\lambda^{h}_{i1}\leqslant\cdots\leqslant\lambda^{h}_{id_{i}},\quad i=1,2,\ldots,p.
\end{align*}

The following conclusion can be found in \cite{d2018optimization, knyazev1985sharp}.

\begin{proposition}\label{thm:cluster_eigen}
For the eigenvalue problem \cref{eq:weak_form_leq} and its finite dimensional discretization \cref{eq:fd_weak_form_leq},  there holds that
\begin{align*}
    0\leqslant\lambda^{h}_{ij}-\lambda_{i}\leqslant\lambda^{h}_{ij}\operatorname{dist}^{2}(\bigoplus_{i=1}^{q}M(\lambda_{i}), V^{h}),\quad \forall (1,1)\leqslant(i,j)\leqslant(q,d_{q}).
\end{align*}
\end{proposition}

Given the eigenvalue problem \cref{eq:weak_form_leq} and its finite dimensional approximation \cref{eq:fd_weak_form_leq}, the following result is classical and can be found in \cite{babuvska1989finite, chatelin2011spectral, knyazev1985sharp}. 
\begin{proposition}\label{thm:cluster_eigenfunc}
	Let $\{u_{ij}^{h}\}$ be the solutions of \cref{eq:fd_weak_form_leq}. There exists $\epsilon_{*}\in (0,1)$ satisfying that if $\operatorname{dist}\left(\bigoplus_{i=1}^{q}M(\lambda_{i}), V^{h}\right)\leqslant \epsilon_{*}$, there exists $\hat{u}_{ij}\in M(\lambda_{i})$ such that 
  \begin{align*}
      \left\Vert u^{h}_{ij}-\hat{u}_{ij}\right\Vert\leqslant C_{*}\operatorname{dist}(\bigoplus_{i=1}^{q}M(\lambda_{i}), V^{h}),\quad \forall (1,1)\leqslant(i,j)\leqslant(q,d_{q}),
  \end{align*}
  where $C_{*}$ is a constant that is independent of $V^{h}$.
\end{proposition}
  \begin{remark}\label{rem:chosen}
  Based on Theorem 1 and Theorem 2 in \cite{knyazev1985sharp}, the constants in \cref{thm:cluster_eigenfunc} can be chosen as 
  	   \begin{align*}
  	       \epsilon_{*}=\alpha \min_{i=1,\ldots,q+1}\sqrt{\frac{\lambda_{i}-\lambda_{i-1}}{\lambda_{i}}} \quad\text{and}\quad C_{*}=\max_{i=1,\ldots,q}	\sqrt{\frac{2(\lambda_{i+1}-\lambda_{i-1})\lambda_{i}}{(1-\alpha^{2})(\lambda_{i}-\lambda_{i-1})(\lambda_{i+1}-\lambda_{i})}}
  	   \end{align*}
  	 with $\alpha\in(0,1)$.
  	\end{remark}

 \section{Quasi-orthogonality}
To understand the philosophy behind ParO and carry out the numerical analysis, we introduce a quasi-orthogonality, which plays a crucial role in our
investigations on approximations of eigenvalue problems.

\begin{definition}
Given $\delta>0$ and $n\geqslant2$, $\{v_j\}^n_{j=1}\subset H$ with $\Vert v_{j}\Vert=1$ for $j=1,\ldots,n$ is said to be $\delta$-quasi-orthogonal if there exists $\{u_j\}^n_{j=1}\subset H$ satisfying that 
\begin{align}\label{ine:orth_distan}
          a(u_{i}, u_{j})=\delta_{ij}, ,  \quad 
          \Vert u_{j}-v_{j}\Vert \leqslant\delta,\quad i,j=1,2,\ldots,n.
      \end{align}
\end{definition}

 The following proposition tells the approximation property of the orthogonalization of quasi-orthogonal vectors. The proof is provided in \cref{sec:p_app_pro_orth}.
 \begin{proposition}\label{prop:app_pro_orth}
     If  $\{v_{j}\}_{j=1}^{n}\subset H$ is $\delta$-quasi-orthogonal, then  there exists an orthonormal basis  $\{\tilde{v}_{j}\}_{j=1}^{n}$ of $\operatorname{span}\{v_{1},\ldots,v_{n}\}$ such that 
         \begin{align*}
          \left\Vert\tilde{v}_{j}-v_{j}\right\Vert\leqslant\sqrt{n}\delta,\quad j=1,\ldots,n.
      \end{align*}
 \end{proposition}

   The following conclusion, which can be derived directly by a triangle inequality from \cref{prop:app_pro_orth}, shows the approximation property between orthonormal bases.
 \begin{corollary}\label{cor:vec}
  If  $\{v_{j}\}_{j=1}^{n}\subset H$ is $\delta$-quasi-orthogonal, i.e., there exists $\{u_{j}\}_{j=1}^{n}\subset H$ satisfying 
 \begin{align*}
         a(u_{i}, u_{j})=\delta_{ij},\quad \Vert u_{j}-v_{j}\Vert \leqslant\delta,\quad~ i,j=1,2,\ldots,n,
      \end{align*}
  then  there exists $\{\tilde{v}_{j}\}_{j=1}^{n}\subset\operatorname{span}\{v_{1},\ldots,v_{n}\}$ such that 
         \begin{align*}
          a(\tilde{v}_{i}, \tilde{v}_{j})&=\delta_{ij},\quad
          \operatorname{dist}(u_{j},\tilde{v}_{j})\leqslant\left\Vert u_{j}-\tilde{v}_{j}\right\Vert\leqslant(\sqrt{n}+1)\delta,\quad i,j=1,\ldots,n.
      \end{align*}
 \end{corollary}
 
\subsubsection*{Orthogonal basis approximation}We arrive at the following proposition from \cref{cor:vec}, which tells the approximation property of orthonormal bases by the distance from one subspace to another and will play a crucial role in our analysis. 
 \begin{proposition}\label{prop:subspace_angle}
    Given $\varepsilon\in(0,1)$ and two subspaces $U, V\subset H$ with $ \operatorname{dim}(U)=\operatorname{dim}(V)=n$. If $ \operatorname{dist}(U, V)\leqslant\varepsilon,$
     then for any orthonormal basis $\{u_{j}\}_{j=1}^{n}$ of $U$, there exists an orthonormal basis $\{w_{j}\}_{j=1}^{n}$ of $V$ satisfying
      \begin{align*}
          \operatorname{dist}(u_{i}, w_{i})\leqslant (1+\sqrt{n})\sqrt{2-2\sqrt{1-\varepsilon^{2}}},\quad i,j=1,2,\ldots,n,
      \end{align*}
 \end{proposition} 
\begin{proof}
First, we show that $\mathcal{P}_{V}|_{U}$ is an isomorphism from $U$ to $V$. Indeed, for $\tilde{u}, \hat{u}\in U$ satisfying $\mathcal{P}_V \tilde{u}=\mathcal{P}_V\hat{u}$, we obtain from
 \begin{align*}
     a(\tilde{u}-\mathcal{P}_{V}\tilde{u}, v)=0, \quad v\in V,\\
     a(\hat{u}-\mathcal{P}_{V}\hat{u}, v)=0, \quad v\in V
 \end{align*}
 that 
 \begin{align*}
     a(\tilde{u}-\hat{u}, v)=0, \quad v\in V,
 \end{align*}
 and $\tilde{u}=\hat{u}$ due to $\operatorname{dist}(U, V)\leqslant\varepsilon<1$. Hence, $\mathcal{P}_{V}|_{U}$ is an  injection and then is isomorphism from $U$ to $V$ since $\operatorname{dim}(V)=\operatorname{dim}(U)$.
 
 Next we set $v_{j}=\frac{\mathcal{P}_{V}u_{j}}{\left\Vert\mathcal{P}_{V}u_{j}\right\Vert}$ for $j=1,2,\ldots,n$. Since $\mathcal{P}_{V}$ is an isomorphism,  we have that  $V=\operatorname{span}(\{v_{j}\}_{j=1}^{n})$ and 
\begin{align*}
    &\mathrel{\phantom{=}}\Vert u_{j}-v_{j}\Vert=\sqrt{\Vert u_{j}-\mathcal{P}_{V}u_{j}\Vert^{2}+\left\Vert\mathcal{P}_{V}u_{j}-\frac{\mathcal{P}_{V}u_{j}}{\Vert\mathcal{P}_{V}u_{j}\Vert}\right\Vert^{2}}\\
    &\leqslant\sqrt{\operatorname{dist}^{2}(U, V)+\left(1-\sqrt{1-\operatorname{dist}^{2}(U, V)}\right)^{2}}\leqslant\sqrt{2-2\sqrt{1-\varepsilon^{2}}},
\end{align*}
i.e., $\{v_{j}\}_{j=1}^{n}$ is $\sqrt{2-2\sqrt{1-\varepsilon^{2}}}$-quasi-orthogonal. Then by \cref{cor:vec}, we complete the proof.
\end{proof}

\subsubsection*{Dimension-preserving} For $V^{h}=\bigoplus_{i=1}^{p}X_{i}$, we consider subspaces $X, Y\subset V^{h}\subset H$ with decompositions as follows:
 \begin{align*}
     X=\bigoplus_{i=1}^{q}X_{i}, \quad Y=\sum_{i=1}^{q}Y_{i},
 \end{align*}
 where $\operatorname{dim}(X)=N$ and $\operatorname{dim}(X_{i})=\operatorname{dim}(Y_{i})=d_{i}$. Obviously,  $\operatorname{dim}(Y)\leqslant\sum_{i=1}^{q}\operatorname{dim}(Y_{i})=N$. We have the following conclusion telling when $\operatorname{dim}(Y)=N$ holds
 and will be used in our analysis. The proof is given in \cref{sec:p_fram_keep_dim}.

\begin{proposition}\label{thm:fram_keep_dim}
Given $\varepsilon\in(0,\min_{1\leqslant i\leqslant q}\sqrt{\frac{4(1+\sqrt{d_{i}})^{2}N-1}{4(1+\sqrt{d_{i}})^{4}N^{2}}})$. If
$$\max_{i=1,\ldots,q}\operatorname{dist}(X_{i}, Y_{i})<\varepsilon,$$ then
\begin{align}
    Y=\bigoplus_{i=1}^{q}Y_{i}.
\end{align}
\end{proposition}

\subsubsection*{An application of quasi-orthogonality} Note that \cref{thm:cluster_eigenfunc} tells only that each orthonormal eigenfunction of \cref{eq:fd_weak_form_leq} approximates some eigenfunction of \cref{eq:weak_form_leq}. However,  these exact eigenfunctions of \cref{eq:weak_form_leq} may not be orthonormal to each other (see Corollary 2.11 in \cite{dai2015convergence}). In practical applications, the approximate property of the approximate eigenfunctions to the exact orthonormal eigenfunctions is usually required, which is for structure-preserving and preventing the accumulation of errors of approximations.
 
From solving $\cref{eq:fd_weak_form_leq}$, we may obtain orthogonal or quasi-orthogonal (when the algebraic error of the orthogonalization is taken into account) normalized basis $\{u_{ij}^{h}\}$ of $M_{h}(\lambda_{i})$. With the investigation of the quasi-orthogonality, we are able to estimate the approximation error of the approximate eigenfunctions to the exact orthonormal eigenfunctions.

\begin{theorem}\label{prop:clus_to_vec}
Let $\{u_{ij}^{h}\}$ be solutions of \cref{eq:fd_weak_form_leq}. 
If $\operatorname{dist}\left(\bigoplus_{i=1}^{q}M(\lambda_{i}), V^{h}\right)\leqslant \epsilon_{*}$, then there exists an orthonormal basis $\{u^{o}_{ij}\}$ of $M(\lambda_{i})$  with $b(u^{o}_{ij}, u^{o}_{kl})=\delta_{ik}\delta_{jl}$ such that 
     \begin{align*}
         \operatorname{dist}\left(u^{o}_{ij},
         u^{h}_{ij}\right)\leqslant C_{**}\operatorname{dist}\left(\bigoplus_{i=1}^{q}M(\lambda_{i}), V^{h}\right), \quad (1,1)\leqslant(i,j)\leqslant(q,d_{q}),
     \end{align*}
     where $C_{**}=\max_{i=1,\cdots,q}\frac{2(\sqrt{d_{i}}+1)}{\sqrt{\lambda}_{i}}C_{*}$. The constants $\epsilon_{*}$ and $C_{*}$ are defined in \cref{thm:cluster_eigenfunc}, and $C_{*}$ is independent of  $V^{h}$.
\end{theorem}
\begin{proof}
For the approximate eigenpairs $\left\{(\lambda^{h}_{ij}, u^{h}_{ij})\right\}_{(1,1)\leqslant(i,j)\leqslant(q,d_{q})}$, we obtain from \cref{thm:cluster_eigenfunc} that there exist $\hat{u}_{ij}\in M(\lambda_{i})$ and the constant $C_{*}$ that is defined in \cref{thm:cluster_eigenfunc} and is independent of $V^{h}$, such that 
\begin{align*}
      \left\Vert u^{h}_{ij}-\hat{u}_{ij}\right\Vert\leqslant C_{*}\operatorname{dist}\left(\bigoplus_{i=1}^{q}M(\lambda_{i}), V^{h}\right),\quad (1,1)\leqslant(i,j)\leqslant(q,d_{q}),
  \end{align*}
which together with the fact $\|u_{ij}^{h}\|_{a}=\sqrt{\lambda_{ij}^{h}}$ and $\lambda_{ij}^{h}\geqslant\lambda_{i}$ yields that
  \begin{align*}
        &\left\Vert\frac{u^{h}_{ij}}{\Vert u^{h}_{ij}\Vert}-\frac{\hat{u}_{ij}}{\Vert \hat{u}_{ij}\Vert}\right\Vert=\frac{1}{\sqrt{\lambda^{h}_{ij}}}\left\Vert u^{h}_{ij}-\hat{u}_{ij}+\left(1-\frac{\Vert u^{h}_{ij}\Vert}{\Vert \hat{u}_{ij}\Vert}\right)\hat{u}_{ij}\right\Vert\\\leqslant&\frac{1}{\sqrt{\lambda}_{i}}\left(\left\Vert u^{h}_{ij}-\hat{u}_{ij}\right\Vert+\left|\Vert \hat{u}_{ij}\Vert-\Vert u^{h}_{ij}\Vert\right|\right)\leqslant\frac{2}{\sqrt{\lambda}_{i}}C_{*}\operatorname{dist}\left(\bigoplus_{i=1}^{q}M(\lambda_{i}), V^{h}\right).
  \end{align*}

Then $\{\frac{\hat{u}_{ij}}{\Vert \hat{u}_{ij}\Vert}\}_{j=1}^{d_{i}}$ is $\frac{2}{\sqrt{\lambda}_{i}}C_{*}\operatorname{dist}\left(\bigoplus_{i=1}^{q}M(\lambda_{i}), V^{h}\right)$-quasi-orthogonal and  it follows from \cref{cor:vec} that there exists an orthonormal basis $\{u^{o}_{ij}\}$ of $M(\lambda_{i})$  with $b(u^{o}_{ij}, u^{o}_{ik})=\delta_{jk}$ such that 
     \begin{align*}
         \operatorname{dist}\left(u^{o}_{ij},
         u^{h}_{ij}\right)\leqslant C_{**}\operatorname{dist}\left(\bigoplus_{i=1}^{q}M(\lambda_{i}), V^{h}\right),\quad j=1,\ldots, d_{i},
     \end{align*}
    where $C_{**}=\max_{i=1,\cdots,q}\frac{2(\sqrt{d_{i}}+1)}{\sqrt{\lambda}_{i}}C_{*}$. We complete the proof.
 \end{proof}
 
 We see that \cref{prop:clus_to_vec} tells that, for the finite dimensional approximation of an eigenvalue problem, there exists a set of orthonormal eigenfunctions whose distance to the orthonormal approximate eigenfunctions is controlled by the distance of subspaces.
 
 In the next section, we will present the approximation of iterative solutions $\left\{\left(\lambda_{ij}^{(n)}, u_{ij}^{(n)}\right)\right\}$ produced by ParO to the solutions of the discrete problem \cref{eq:fd_weak_form_leq} using \cref{prop:clus_to_vec}. Then we obtain the approximation errors of iterative solutions to solutions of \cref{eq:weak_form_leq} from \cref{thm:cluster_eigen}, \cref{prop:clus_to_vec}
and the triangle inequalities
 \begin{align*}
     \left|\lambda_{i}-\lambda_{ij}^{(n)}\right|&\leqslant\left|\lambda_{i}-\lambda^{h}_{ij}\right|+\left|\lambda^{h}_{ij}-\lambda_{ij}^{(n)}\right|,\\\operatorname{dist}(u_{ij}^{o}, u_{ij}^{(n)})&\leqslant\operatorname{dist}(u_{ij}^{o}, u^{h,o}_{ij})+\operatorname{dist}(u^{h,o}_{ij}, u_{ij}^{(n)}),
 \end{align*}
for $(1,1)\leqslant(i,j)\leqslant(q,d_{q})$, where
$\{u^{o}_{ij}\}_{j=1}^{d_{i}}$ and $\{u^{h,o}_{ij}\}_{j=1}^{d_{i}}$ with $b(u^{o}_{ij}, u^{o}_{kl})=b(u^{h,o}_{ij}, u^{h,o}_{kl})=\delta_{ik}\delta_{jl}$ are orthogonal bases of $M(\lambda_{i})$ and $M_{h}(\lambda_{i})$, respectively.
 
\section{Convergence and Error Estimates}With the quasi-orthogonality, in this section, we are able to obtain the convergence and error estimates of the numerical approximations produced by ParO for clustered eigenvalue problems.

\subsection{Algorithm framework}
We first recall the framework of ParO for the first $N$ clustered eigenvalues and their corresponding eigenfunctions of \cref{eq:weak_form_leq}, which is stated as \cref{algo:fram_matr}. We mention that \cref{algo:fram_matr} is indeed a modified version of Algorithm 1.1 in \cite{dai2014parallel} (see \cref{flow_fram} for its flowchart). We see form \cref{flow_fram} clearly that our framework has a feature of 
two level parallelization in Step 2: one level is obtained by updating $N$ eigenfunctions intrinsically in parallel; the other level is obtained by applying the standard parallel strategies when updating each eigenfunction.
\begin{algorithm}[htbp]
\caption{A framework for ParO}\label{algo:fram_matr}
\begin{enumerate}
\item Given a finite dimensional subspace $V^{h}$ and initial data $\left(\lambda_{k}^{(0)}, u_{k}^{(0)}\right)\in\mathbb{R}\times V^{h}$ for $k=1, 2, \ldots, N$, let $n=0$.
\item For $k=1,2,\ldots, N$, update $u_{k}^{(n)}$ to obtain $u_{k}^{(n+1/2)}$ in parallel.
\item Let $u_{k}^{(n+1)}=u_{k}^{(n+1/2)}$ for $k=1,\ldots, N$. Or if necessary, project \cref{eq:weak_form_leq} onto $U_{n+1}=\operatorname{span}\left\{u_{1}^{(n+1/2)}, \ldots, u_{N}^{(n+1/2)}\right\}$ and obtain eigenpairs $\left(\lambda_{k}^{(n+1)}, u_{k}^{(n+1)}\right)$. 
\item If not converge, let $n=n+1$ and go to Step $2$.
\end{enumerate}
\end{algorithm}

\begin{figure}[htbp!]
	\centering
    \includegraphics[width=13cm]{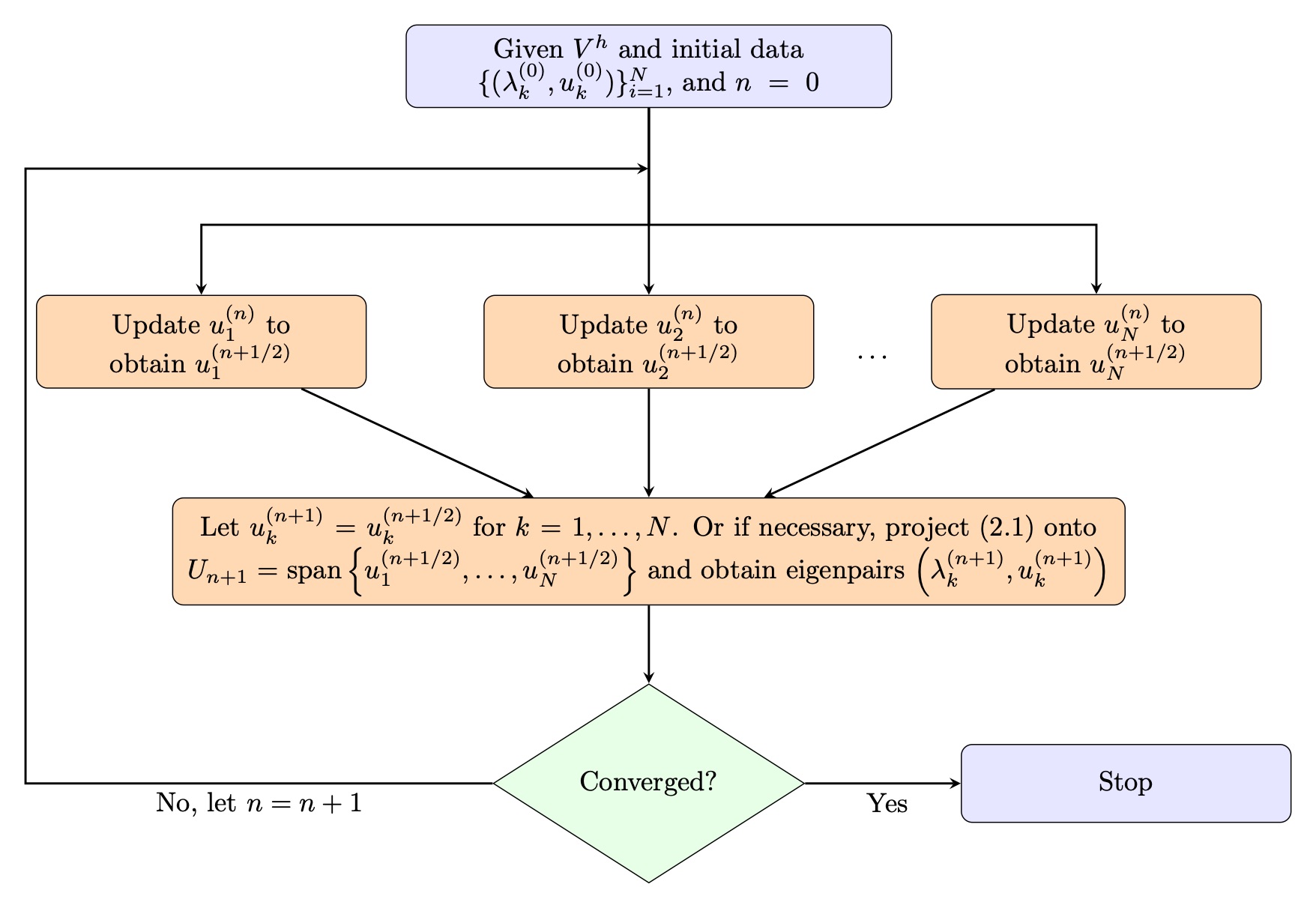}
 	\caption{Flowchart of \cref{algo:fram_matr}}\label{flow_fram}
\end{figure}

\subsubsection*{Step 1} We see that there are several ways to provide initial data in step 1 of Algorithm 4.1. We may obtain initial data from
\begin{itemize}
\item physical observation or data (see, e.g., \cite{dai2014parallel, payne1992iterative}). For instance, in electronic structure calculations, we may choose initial data from Gaussian-type orbital, Slater-type orbital, atomic orbital based guesses, and so on;
    \item solving the eigenvalue problem on a coarse grid; 
\item neural networks based guesses \cite{hazra2024predicting, xu2024subspace}.
\end{itemize}

Since we look for clustered eigenvalues and their corresponding eigenfunctions, we shall consider the approximation of each eigenspace as mentioned in Introduction. We understand that it is not trivial to obtain the multiplicity $d_{i}$ of each eigenvalue $\lambda_{i}$. A possible way to approximate the multiplicities is to cluster the initial guesses $\lambda_{1}^{(0)}\leqslant\lambda_{2}^{(0)}\leqslant\cdots\leqslant\lambda_{N}^{(0)}$. By clustering methods such as Bayesian Information Criterion and Silhouette method (see e.g., \cite{rousseeuw1987silhouettes, schwarz1978estimating}), we can get $q'$ clusters with $d_{i}'$ eigenpairs in the $i-$th cluster $(i=1,\ldots,q')$, that is,  
\begin{align}\label{eq:clustering_eigenpairs}
    \left\{\left(\lambda_{ij}^{(0)}, u_{ij}^{(0)}\right)\right\}_{i=1,\ldots,q',j=1,\ldots,d_{i}'}= \left\{\left(\lambda_{k}^{(0)}, u_{k}^{(0)}\right)\right\}_{k=1,\ldots,N}.
\end{align}

 In this paper, we assume that $q'=q$ and $d_{i}'=d_{i}$ for $i=1,\ldots,q'$. Such an assumption is likely to hold since the previously mentioned methods may be able to give good initial data. We also want to point out that, the numerical experiments in \cite{dai2014parallel, pan2017parallel} have shown that such an assumption is unnecessary in practical computations.

Define $U_{0}=\operatorname{span}\left\{u_{1}^{(0)}, u_{2}^{(0)}, \ldots, u_{N}^{(0)}\right\}$ and 
\begin{align*}
    \lambda_{ij}^{(n+1)}:=\lambda_{\sum_{r=0}^{i-1}d_{r}+j}^{(n+1)},\quad u_{ij}^{(n+1)}:=u_{\sum_{r=0}^{i-1}d_{r}+j}^{(n+1)},\quad (1,1)\leqslant(i,j)\leqslant(q,d_{q}), \quad n\geqslant0.
\end{align*}
Set 
\begin{align}\label{def:pieces}
    U_{n}^{(i)}=\operatorname{span}\left\{u_{i1}^{(n)}, \ldots, u_{id_{i}}^{(n)}\right\},\quad i=1,\ldots, q,
\end{align}
then $U_{n}=\sum_{i=1}^{q} U_{n}^{(i)}$. 

\subsubsection*{Step 2} To update each eigenfunction in step $2$ of \cref{algo:fram_matr}, as pointed out in \cite{dai2014parallel}, we can apply the shifted-inverse approach, Chebyshev filtering, and so on. Define $ \mathcal{F}_{n}: U_{n}\rightarrow U_{n+1}$ as follows
\begin{align*}
    \mathcal{F}_{n}^{(i)}:=\mathcal{F}_{n}|_{U_{n}^{(i)}}:u_{ij}^{(n)}=u_{\sum_{r=0}^{i-1}d_{r}+j}^{(n)}\mapsto u_{\sum_{r=0}^{i-1}d_{r}+j}^{(n+1/2)}:= u_{ij}^{(n+1/2)},\quad (1,1)\leqslant(i,j)\leqslant(q,d_{q}).
\end{align*}

Set
\begin{align*}
    U_{n+1/2}^{(i)}:=\{u^{(n+1/2)}_{i1}, \dots, u^{(n+1/2)}_{id_{i}}\},\quad i=1,\ldots,q.
\end{align*}
We have $U_{n+1}=\sum_{i=1}^{q} U_{n+1/2}^{(i)}$. 
 
\subsubsection*{Step 3}In step $3$ of \cref{algo:fram_matr}, if we choose to update the approximations, we can solve a small scale eigenvalue problem as follows: find $(\lambda^{(n+1)}, u^{(n+1)})\in\mathbb{R}\times U_{n+1}$ satisfying 
	\begin{align}\label{eq:pro_ei_pro}
	    a\left(u^{(n+1)},v\right)=\lambda^{(n+1)}b\left(u^{(n+1)},v\right)\quad \forall v\in U_{n+1},
	\end{align}
to obtain eigenpairs $(\lambda_{ij}^{(n+1)}, u_{ij}^{(n+1)})$ with
$b\left(u_{ij}^{(n+1)}, u_{kl}^{(n+1)}\right)=\delta_{ik}\delta_{jl}$ for $(1,1)\leqslant(i,j), (k,l)\leqslant(q,d_{q})$. 

\subsubsection*{Error estimate} The following theorem shows the approximation errors of eigenpairs 
 when solving a small scale eigenvalue problem is carried out under the assumption that eigenfunctions are approximated well to a certain in Step 2. The proof is given in \cref{prof:conv_fram}.
 
\begin{theorem}\label{prop:conv_fram} If $\operatorname{dist}(M_{h}(\lambda_{i}), U_{n+1/2}^{(i)})\ll1 (i=1,\ldots,q)$, then after solving a small scale eigenvalue problem \cref{eq:pro_ei_pro} in $U_{n+1}=\sum_{i=1}^{q}U_{n+1/2}^{(i)}$, there exists an orthonormal basis $\{u^{h,o}_{ij}\}_{j=1}^{d_{i}}$ of $M_{h}(\lambda_{i})$ with $b(u^{h,o}_{ij}, u^{h,o}_{kl})=\delta_{ik}\delta_{jl}$ such that for $(i, j)\leqslant(q, d_{q})$
\begin{align*}
    |\lambda_{ij}^{(n+1)}-\lambda^{h}_{ij}|\leqslant&\lambda_{i+1,1}^{h} \operatorname{dist}^{2}\left(\bigoplus_{i=1}^{q}M_{h}(\lambda_{i}), U_{n+1}\right)\leqslant\lambda_{i+1,1}^{h}\sqrt{N}\max_{1\leqslant i\leqslant q}\operatorname{dist}^{2}(M_{h}(\lambda_{i}), U_{n+1/2}^{(i)}),\\
    \operatorname{dist}(u^{h,o}_{ij}, u_{ij}^{(n+1)})\leqslant& \tilde{C}_{**}\operatorname{dist}\left(\bigoplus_{i=1}^{q}M_{h}(\lambda_{i}), U_{n+1}\right)\leqslant \tilde{C}_{**}\sqrt{N}\max_{1\leqslant i\leqslant q}\operatorname{dist}(M_{h}(\lambda_{i}), U_{n+1/2}^{(i)}),
\end{align*}
where the constant $\tilde{C}_{**}$ is independent of $U_{n+1/2}$.
\end{theorem}

We will see in the next subsection that the requirements
$\operatorname{dist}(M_{h}(\lambda_{i}), U_{n+1/2}^{(i)})\ll1 (i=1,\ldots,q)$ can be achieved. We understand that the generation and solution of the $N$-dimensional eigenvalue problem in Step $3$ of \cref{algo:fram_matr} is computationally expensive. Fortunately, we see from the numerical experiments in \cite{dai2014parallel, pan2017parallel} that the resulting matrix is almost diagonal: the non-diagonal entries are very small and the computational cost is not very 
high. 
 \subsection{Shifted-inverse based ParO algorithm} In this subsection, we shall study the ParO algorithm for solving the clustered eigenvalue problem when the shifted-inverse approach is applied to update each eigenfunction. The shifted-inverse based ParO algorithm (\cref{algo:pou_n_shifted}), which solves a small scale eigenvalue problem in each iteration to update the shifts, has been proposed in \cite{dai2014parallel, pan2017parallel}. To show the convergence, we first consider a simplified version (\cref{algo:n_shif}) which fixes the shifts and does not carry out the steps of solving projected eigenvalue problems in iterations. Based on the numerical analysis of the simplified version, we then prove that the approximations produced by \cref{algo:pou_n_shifted} converge rapidly. 
 
\subsubsection*{Simplified shifted-inverse based ParO algorithm} We set the shift as any convex combination of $\left\{\lambda_{ij}^{(0)}\right\}_{j=1}^{d_{i}}$ denoted by
  \begin{align*}
      \bar{\lambda}_{i}:=\mathcal{C}_{i}\left(\left\{\lambda_{ij}^{(0)}\right\}_{j=1}^{d_{i}}\right), \quad i=1,\ldots,q.
  \end{align*}
 For instance, we can choose $\mathcal{C}_{i}\left(\left\{\lambda_{ij}^{(0)}\right\}_{j=1}^{d_{i}}\right)=\frac{1}{d_{i}}\sum_{j=1}^{d_{i}}\lambda_{ij}^{(0)}$. 
 
 In our discussion, we assume that the shifts are always not equal to any eigenvalue of \cref{eq:fd_weak_form_leq} in the calculation process. Otherwise, we continue the iterative process on other eigenfunctions while keeping the eigenfunctions unchanged. 
 
 The shifted-inverse approach $\mathcal{F}_{n}$ writes: for $U_{n}^{(i)}=\operatorname{span}\left\{u_{i1}^{(n)}, \ldots, u_{id_{i}}^{(n)}\right\}, $ $\mathcal{F}_{n}^{(i)}:=\mathcal{F}_{n}|_{U_{n}^{(i)}}:U_{n}^{(i)}\rightarrow U_{n+1/2}^{(i)}$ with $u_{ij}^{(n+1/2)}=\mathcal{F}_{n}^{(i)}u_{ij}^{(n)}$ for $i=1,\ldots,q, j=1,2,\ldots,d_{i}$ satisfying
 \begin{align*}
    a(u_{ij}^{(n+1/2)}, v)-\bar{\lambda}_{i}b(u_{ij}^{(n+1/2)}, v)=\bar{\lambda}_{i}b(u_{ij}^{(n)}, v), \quad \forall v\in V^{h}.
 \end{align*}
 
The simplified shifted-inverse based ParO algorithm is stated as \cref{algo:n_shif}. Compared with \cref{algo:fram_matr}, step 3 is no longer carried out in this simplified version. 

\begin{algorithm}[htbp]
\caption{Simplified shifted-inverse based ParO algorithm}\label{algo:n_shif}
\begin{enumerate}
\item Given a finite dimensional space $V^{h}$ and $tol>0$, provide and cluster initial data by \cref{eq:clustering_eigenpairs}, i.e.,  $\left\{\left(\lambda_{ij}^{(0)}, u_{ij}^{(0)}\right)\right\}_{(1,1)\leqslant(i,j)\leqslant(q,d_{q})}\subset\mathbb{R}\times V^{h}$. Set $\bar{\lambda}_{i}=\mathcal{C}_{i}\left(\left\{\lambda_{ij}^{(0)}\right\}_{j=1}^{d_{i}}\right)$ and let $n=0$.

\item For $ (1,1)\leqslant(i,j)\leqslant(q,d_{q})$, find $u_{ij}^{(n+1/2)}\in V^{h}$ in parallel by solving  
	\begin{align}\label{eq:inv_pow}
	    a\left(u_{ij}^{(n+1/2)},v\right)-\bar{\lambda}_{i}b\left(u_{ij}^{(n+1/2)},v\right)=\bar{\lambda}_{i}b\left(u_{ij}^{(n)},v\right)\quad\forall v\in V^{h}.
	\end{align}

\item Set $u_{ij}^{(n+1)}=\frac{u_{ij}^{(n+1/2)}}{\left\Vert u_{ij}^{(n+1/2)}\right\Vert_{b}}$. If $\frac{\Vert u_{ij}^{(n+1)}-u_{ij}^{(n)}\Vert_{b}}{\Vert u_{ij}^{(n)}\Vert_{b}}>tol$, let $n=n+1$ and go to Step $2$.
\end{enumerate}
\end{algorithm}

If the initial guesses approximate the exact eigenvalues well enough, then the source problems \cref{eq:inv_pow} will be ill-conditioned. We mention that there are approaches to deal with ill-conditioned systems (see e.g., \cite{axelsson1996iterative, bai2006shift, bai2002regularized, riley1955solving, saad2003iterative}). Indeed, it is quite difficult to solve these ill-conditioned systems well, which will be discussed in our other work. Here we assume that such systems can be well solved.

 \cref{algo:n_shif} may be viewed as an extension of the shifted-inverse approach to clustered eigenvalue problems. 
 
 The following proposition tells that the approximation to the eigenspace in the iteration process is of dimension-preserving. The proof is given in \cref{sec:p_multi_keep_dim}.
   
    \begin{proposition}\label{thm:multi_keep_dim}
      If $U^{(i)}_{n+1/2}$ is obtained by  \cref{algo:n_shif}, then
\begin{align*}
    \operatorname{dim}\left(U_{n+1/2}^{(i)}\right)=\operatorname{dim}\left(U_{n}^{(i)}\right),\quad i=1,2,\ldots,q.
\end{align*}
     \end{proposition}
     
 With the help of \cref{thm:multi_keep_dim}, we obtain convergence of eigenspace approximations produced by \cref{algo:n_shif}. The proof is given in \cref{sec:p_n_shifted_multi}.
    
      \begin{theorem}\label{thm:n_shifted_multi}
     Assume that
    \begin{align}\label{ine:eig_guess_theta_multi}
       0<\delta_{0}:=\max_{(1,1)\leqslant(i,j)\leqslant(q,d_{q})}|\lambda^{h}_{ij}-\bar{\lambda}_{i}|<\frac{g}{2},
    \end{align}
   where $g:=\min_{1\leqslant i<r\leqslant q+1}|\lambda^{h}_{id_{i}}-\lambda^{h}_{r1}|;$
   \begin{align}\label{ine:init_theta_multi_keep_dim_1}
        \operatorname{dim}\left(U_{0}^{(i)}\right)=d_{i},\quad \operatorname{dist}\left(M_{h}(\lambda_{i}), U_{0}^{(i)}\right)<1\quad\forall i=1,2,\ldots,q.
    \end{align}
If $U_{n+1/2}^{(i)}$ is produced by \cref{algo:n_shif}, then
\begin{align*}
    \operatorname{dist}\left(M_{h}(\lambda_{i}), U_{n+1/2}^{(i)}\right)\leqslant\varepsilon_{n+1}:=\frac{\delta_{0}\varepsilon_{n}}{\sqrt{(g-\delta_{0})^{2}(1-\varepsilon_{n}^{2})+\delta_{0}^{2}\varepsilon^{2}_{n}}},\quad \lim_{n\rightarrow\infty}\frac{\varepsilon_{n+1}}{\varepsilon_{n}}=\frac{\delta_{0}}{g-\delta_{0}}.
\end{align*}
    \end{theorem}

    \cref{thm:n_shifted_multi} shows that convergence of 
    \cref{algo:n_shif} does not require sufficiently accurate initial guesses. \cref{ine:eig_guess_theta_multi} ensures that the shift is closer to the eigenvalue being approximated. Indeed, \cref{ine:eig_guess_theta_multi} can be satisfied 
    under the assumption that the finite dimensional discretization \cref{eq:fd_weak_form_leq} approximates \cref{eq:weak_form_leq} not ``too badly". \cref{ine:init_theta_multi_keep_dim_1} guarantees that the dimension of the approximated subspace is preserved. When $N=1$, \cref{lim:depen} may be reviewed as the classical shift-inverse convergence result.
    
\subsubsection*{Shifted-inverse based ParO algorithm} We now analyze the convergence of the shifted-inverse based ParO algorithm proposed in \cite{dai2014parallel, pan2017parallel}, which is stated as \cref{algo:pou_n_shifted}.

\begin{algorithm}[H]
\caption{Shifted-inverse based ParO algorithm}\label{algo:pou_n_shifted}
\begin{enumerate}
\item Given a finite dimensional space $V^{h}$ and $tol>0$, provide and cluster initial data by \cref{eq:clustering_eigenpairs}, i.e.,  $\left\{\left(\lambda_{ij}^{(0)}, u_{ij}^{(0)}\right)\right\}_{(1,1)\leqslant(i,j)\leqslant(q,d_{q})}\subset\mathbb{R}\times V^{h}$. Set $\bar{\lambda}_{i}^{(0)}=\mathcal{C}_{i}\left(\left\{\lambda_{ij}^{(0)}\right\}_{j=1}^{d_{i}}\right)$ and let $n=0$.
\item For $(1,1)\leqslant(i,j)\leqslant(q,d_{q})$, find $u_{ij}^{(n+1/2)}\in V^{h}$ in parallel by solving   
	\begin{align}\label{eq:pou_n_shifted_inv_pow}
	    a\left(u_{ij}^{(n+1/2)},v\right)-\bar{\lambda}_{i}^{(n)}b\left(u_{ij}^{(n+1/2)},v\right)=\bar{\lambda}_{i}^{(n)}b\left(u_{ij}^{(n)},v\right)\quad\forall v\in V^{h}.
	\end{align}
\item Solve an eigenvalue problem: find $(\lambda^{(n+1)}, u^{(n+1)})\in\mathbb{R}\times U_{n+1}=\mathbb{R}\times\operatorname{span}\left\{u_{11}^{(n+1/2)}, \ldots, u_{qd_{q}}^{(n+1/2)}\right\}$ satisfying
	\begin{align}\label{eq:pou_n_shifted_sub}
	    a\left(u^{(n+1)},v\right)=\lambda^{(n+1)}b\left(u^{(n+1)},v\right)\quad \forall v\in U_{n+1},
	\end{align}
to obtain eigenpairs $\left\{\left(\lambda_{ij}^{(n+1)}, u_{ij}^{(n+1)}\right)\right\}$ with
$b\left(u_{ij}^{(n+1)}, u_{kl}^{(n+1)}\right)=\delta_{ik}\delta_{jl}$.
\item If $\sum_{i=1}^{q}\sum_{j=1}^{d_{i}}\left|\lambda_{ij}^{(n+1)}-\lambda_{ij}^{(n)}\right|>tol$, set $\bar{\lambda}_{i}^{(n+1)}=\mathcal{C}_{i}\left(\left\{\lambda_{ij}^{(n+1)}\right\}_{j=1}^{d_{i}}\right)$, $n=n+1$ and go to Step $2$.
\end{enumerate}
\end{algorithm}

As mentioned above, we assume that the shifts are always not equal to any eigenvalue of \cref{eq:fd_weak_form_leq} in the calculation process. If there are cases where some eigenfunctions are very well approximated, while other eigenfunctions have not yet converged, then we continue the iterative process on the non-converged eigenfunctions while keeping the well approximated eigenfunctions unchanged.

The following theorem tells the convergence of \cref{algo:pou_n_shifted}. The proof is given in \cref{sec:p_clus_pou_n_shifted}
 \begin{theorem}\label{prop:clus_pou_n_shifted}Assume that
there exists $0<\varepsilon_{0}\ll1$ and an orthonormal basis $\{u^{h,o,0}_{ij}\}_{j=1}^{d_{i}}$ of $M_{h}(\lambda_{i}) (i=1,\ldots,q)$ with $b(u^{h,o,0}_{ij}, u^{h,o,0}_{kl})=\delta_{ik}\delta_{jl}$ such that 
    \begin{align}
        \operatorname{dist}\left(u^{h,o,0}_{ij}, u_{i j}^{(0)}\right)&\leqslant\varepsilon_{0},\quad (1,1)\leqslant(i,j)\leqslant(q,d_{q}),\\\label{eq:eigensagf}
       \zeta_{0}:=\max_{1\leqslant i\leqslant q}\left|\mathcal{C}_{i}\left(\left\{\lambda^{h}_{ij}\right\}_{j=1}^{d_{i}}\right)-\bar{\lambda}_{i}^{(0)}\right|&\ll g,\quad \gamma:=\max_{1\leqslant i\leqslant q}\left(\lambda^{h}_{id_{i}}-\lambda^{h}_{i1}\right)\ll g.
    \end{align}
If $\{u_{ij}^{(n+1)}\}$ are produced by \cref{algo:pou_n_shifted}, then there exists an orthonormal basis $\{u^{h,o,n+1}_{ij}\}_{j=1}^{d_{i}}$ of $M_{h}(\lambda_{i}) (i=1,\ldots,q)$ with $b(u^{h,o,n+1}_{ij}, u^{h,o,n+1}_{kl})=\delta_{ik}\delta_{jl}$ such that
  \begin{align*}
      \operatorname{dist}\left(u^{h,o,n+1}_{ij}, u_{ij}^{(n+1)}\right)&\leqslant\varepsilon_{n+1}:=\frac{\tilde{C}_{**}\sqrt{DN}\left(\gamma+\zeta_{n}\right)\varepsilon_{n}}{\sqrt{(g-\gamma-\zeta_{n})^{2}(1-D\varepsilon_{n}^{2})+D\left(\gamma+\zeta_{n}\right)^{2}\varepsilon^{2}_{n}}},\\ |\lambda^{h}_{ij}-\lambda_{ij}^{(n+1)}|&\leqslant\zeta_{n+1}:=\frac{\lambda^{h}_{q+1, 1}}{\tilde{C}_{**}^{2}}\varepsilon_{n+1}^{2},\quad (1,1)\leqslant(i,j)\leqslant(q,d_{q}),
  \end{align*}
   where $D:=\max_{1\leqslant i\leqslant q}d_{i}$ and the constant $\tilde{C}_{**}$ is defined in \cref{prop:conv_fram} and is independent of $U_{n} (n=0, 1, 2,\ldots)$, and 
   \begin{align}\label{lim:cubic_results}
         \lim_{n\rightarrow\infty}\frac{\varepsilon_{n+1}}{\varepsilon_{n}}=\frac{\tilde{C}_{**}\sqrt{DN}\gamma}{g-\gamma}.
    \end{align}
   \end{theorem}

Note that $\gamma\ll1$ when  $\operatorname{dist}\left(\bigoplus_{i=1}^{q}M(\lambda_{i}), V^{h}\right)\ll1$, which together with \cref{lim:cubic_results} implies that \cref{algo:pou_n_shifted} converges faster when the finite dimensional discretization \cref{eq:fd_weak_form_leq} approximates \cref{eq:weak_form_leq} better.

If \cref{eq:weak_form_leq} is already a discrete eigenvalue problem, then $\gamma=0$. We obtain from the proof of \cref{prop:clus_pou_n_shifted} that 
\begin{align}\label{lim:cubicdewre}
    \lim_{n\rightarrow\infty}\frac{\varepsilon_{n+1}}{\varepsilon_{n}^{3}}=\frac{\sqrt{DN}\lambda^{h}_{q+1,1}}{g\tilde{C}_{**}},
\end{align}
which implies that it is a cubic convergence result. Note that the above cubic convergence stems from the convergence of the shifts to some eigenvalues of (2.3), which is exactly what $\gamma=0$ means. The constant $\tilde{C}_{**}$ in \cref{lim:cubicdewre} comes from the application of \cref{prop:conv_fram}. Indeed, we can choose a larger $\tilde{C}_{**}$ since \cref{lim:cubicdewre} implies that \cref{algo:pou_n_shifted} 
converges faster with larger $\tilde{C}_{**}$. However, 
when $\tilde{C}_{**}$ is larger, more exact initial values are required. For 
example, if we expect $\varepsilon_{n}$ to decrease towards $0$ as $n\rightarrow\infty$, a necessary condition
\begin{align*}
    \varepsilon_{0}\geqslant\varepsilon_{1}=\frac{\tilde{C}_{**}\sqrt{DN}\zeta_{0}\varepsilon_{0}}{\sqrt{(g-\zeta_{0})^{2}(1-D\varepsilon_{0}^{2})+D\zeta_{0}^{2}\varepsilon^{2}_{0}}},
\end{align*}
implies that $\varepsilon_{0}$ and $\zeta_{0}$ are required smaller with larger $\tilde{C}_{**}$. Hence, from \cref{prop:clus_to_vec} and \cref{lim:cubic_results}, we obtain that \cref{algo:pou_n_shifted} applied to a discrete eigenvalue problem converges in cubic rate and goes faster with more exact initial values. 
The classical result in the discrete case $(N=1)$ is stated as
\begin{align*}
    \lim_{n\rightarrow\infty}\frac{\varepsilon_{n+1}}{\varepsilon_{n}^{3}}\leqslant1,
\end{align*}
under the assumption of convergence of the algorithm (see, e.g. \cite{arbenz2012lecture,parlett1998symmetric}). In contrast, \cref{prop:clus_pou_n_shifted} is more precise and also for general clustered eigenvalue problems. In addition, \cref{lim:cubic_results} and \cref{lim:cubicdewre} show what the speed of the convergence depends on.

\subsubsection*{Modified shifted-inverse based ParO algorithm} To improve the numerical stability, a modified shifted-inverse based ParO algorithm is proposed in \cite{pan2017parallel}. We note that step $2$ of \cref{algo:pou_n_shifted} can also be written as follows: for $ (1,1)\leqslant(i,j)\leqslant(q,d_{q})$, find $e_{ij}^{(n+1/2)}\in V^{h}$ in parallel satisfying   
	\begin{align}
	    a(e_{ij}^{(n+1/2)},v)-\bar{\lambda}_{i}^{(n)}b(e_{ij}^{(n+1/2)},v)=2\bar{\lambda}_{i}^{(n)}b(u_{ij}^{(n)},v)-a(u_{ij}^{(n)},v)\quad\forall v\in V^{h},
	\end{align}
	and set $u_{ij}^{(n+1/2)}=u_{ij}^{(n)}+e_{ij}^{(n+1/2)}$. Then instead of solving the $N$-dimensional projected eigenvalue problem in $U_{n+1}=\operatorname{span}\left\{u_{11}^{(n+1/2)}, \ldots, u_{1d_{1}}^{(n+1/2)}, \ldots, u_{qd_{q}}^{(n+1/2)}\right\}$, we consider the augmented $2N$-dimensional subspace 
	\begin{align*}
	    \tilde{U}_{n+1}=\operatorname{span}\left\{u_{11}^{(n+1/2)}, \ldots, u_{1d_{1}}^{(n+1/2)}, \ldots, u_{qd_{q}}^{(n+1/2)}, e_{11}^{(n+1/2)}, \ldots, e_{1d_{1}}^{(n+1/2)}, \ldots, e_{qd_{q}}^{(n+1/2)}\right\}.
	\end{align*}
	
For the completeness of this paper, we show the modified shifted-inverse based ParO algorithm proposed in \cite{pan2017parallel} here, which is stated as \cref{algo:modi_pou_n_shifted}.  
\begin{algorithm}[H]
\caption{Modified Shifted-Inverse Based ParO Algorithm }\label{algo:modi_pou_n_shifted}
\begin{enumerate}
\item Given a finite dimensional space $V^{h}$ and $tol>0$, provide and cluster initial data by \cref{eq:clustering_eigenpairs}, i.e.,  $\left\{\left(\lambda_{ij}^{(0)}, u_{ij}^{(0)}\right)\right\}_{(1,1)\leqslant(i,j)\leqslant(q,d_{q})}\subset\mathbb{R}\times  V^{h}$. Set $\bar{\lambda}_{i}^{(0)}=\mathcal{C}_{i}\left(\left\{\lambda_{ij}^{(0)}\right\}_{j=1}^{d_{i}}\right)$ and let $n=0$.
\item For $(1,1)\leqslant(i,j)\leqslant(q,d_{q})$, find $e_{ij}^{(n+1/2)}\in V^{h}$ in parallel by solving
	\begin{align*}
	    a(e_{ij}^{(n+1/2)},v)-\bar{\lambda}_{i}^{(n)}b(e_{ij}^{(n+1/2)},v)=2\bar{\lambda}_{i}^{(n)}b(u_{ij}^{(n)},v)-a(u_{ij}^{(n)},v)\quad\forall v\in V^{h}.
	\end{align*}
\item Solve an eigenvalue problem: find $(\lambda^{(n+1)}, u^{(n+1)})\in\mathbb{R}\times \tilde{U}_{n+1}=\mathbb{R}\times\operatorname{span}\left\{u_{11}^{(n+1/2)}, \ldots, u_{qd_{q}}^{(n+1/2)}, e_{11}^{(n+1/2)}, \ldots, e_{qd_{q}}^{(n+1/2)}\right\}$ satisfying 
	\begin{align}
	    a\left(u^{(n+1)},v\right)=\lambda^{(n+1)}b\left(u^{(n+1)},v\right)\quad \forall v\in \tilde{U}_{n+1},
	\end{align}
to obtain eigenpairs $\left\{\left(\lambda_{ij}^{(n+1)}, u_{ij}^{(n+1)}\right)\right\}$ with
$b\left(u_{ij}^{(n+1)}, u_{kl}^{(n+1)}\right)=\delta_{ik}\delta_{jl}$.
\item If $\sum_{i=1}^{q}\sum_{j=1}^{d_{i}}\left|\lambda_{ij}^{(n+1)}-\lambda_{ij}^{(n)}\right|>tol$, set $\bar{\lambda}_{i}^{(n+1)}=\mathcal{C}_{i}\left(\left\{\lambda_{ij}^{(n+1)}\right\}_{j=1}^{d_{i}}\right)$, $n=n+1$ and go to Step $2$.
\end{enumerate}
\end{algorithm}

The convergence of \cref{algo:modi_pou_n_shifted} follows from the similar argument of the proof for \cref{prop:clus_pou_n_shifted} together with the fact that 
\begin{align*}
    \tilde{U}_{n+1}=U_{n+1}\cup U_{n}, \quad \operatorname{dist}(M_{h}(\lambda_{i}), \tilde{U}_{n+1})\leqslant\operatorname{dist}(M_{h}(\lambda_{i}), U_{n+1}).
\end{align*}

In fact, there have been several numerical experiments for ParO. Numerical experiments in \cite{dai2014parallel} apply finite element discretizations and simulate several typical molecular systems: $\ce{H2O}$ (water), $\ce{C9H8O4}$ (aspirin), $\ce{C5H9O2N}$ ($\alpha$ amino acid), $\ce{C20H14N4}$ (porphyrin) and $\ce{C60}$ (fullerene). 
Numerical experiments in \cite{pan2017parallel} apply the plane-wave discretization and simulate several crystalline systems: $\ce{Si}$ (silicon), $\ce{MgO}$ (magnesium oxide), and $\ce{Al}$ (aluminum) with different sizes. These numerical experiments show that ParO is efficient and can produce highly accurate approximations to large scale systems. Good scalability of parallelization of ParO is also shown in those experiments. It is noteworthy that ParO has been integrated into the electronic structure calculation software Quantum ESPRESSO.

\section{Concluding Remarks}   
In this paper, we have provided the numerical analysis of ParO for linear eigenvalue problems based on the investigation of a quasi-orthogonality. Within the framework of ParO, we have demonstrated the convergence of the associated practical algorithms. We point out that numerical experiments in \cite{dai2014parallel, dai2021parallel, pan2017parallel} show that ParO is very efficient for electronic structure calculations. Due to the space limitation, we shall address the numerical analysis of the approach for the Kohn-Sham equation in a separate article. It is also our ongoing work to carry out the numerical analysis for the ParO based optimization approach proposed in \cite{dai2021parallel}.

\section*{Acknowledgments}

The authors would like to thank Prof. Zhaojun Bai, and the anonymous referees for their constructive suggestions for the enhancement of the conclusions in Section 3 and the improvement of the presentation. 

\appendix
\section{Detailed Proofs}\label{appen}
\subsection{Proof of \cref{prop:app_pro_orth}}\label{sec:p_app_pro_orth}
 \begin{proof}
	For convenience, we first introduce the following notation. For $\{x_{i}\}_{i=1}^{n}, \{y_{i}\}_{i=1}^{n}\subset H$, we set $\mathcal{X}=(x_{1}, x_{2}, \ldots, x_{n}),$  $\mathcal{Y}=(y_{1}, y_{2}, \ldots, y_{n})\in H^{n}$ and define 
	\begin{align*}
		\Vert \mathcal{X}\Vert=\sqrt{\sum_{i=1}^{n}\Vert x_{i}\Vert^{2}},\quad\text{and}\quad \mathcal{X}^{T}\mathcal{Y}=\begin{pmatrix}a(x_{1},y_{1})&\cdots&a(x_{1},y_{n})\\\vdots&\ddots&\vdots\\a(x_{n},y_{1})&\cdots&a(x_{n},y_{n})\end{pmatrix}.
	\end{align*}
	
	Let $\mathcal{V}=(v_{1}, v_{2}, \ldots, v_{n})$, since $\{v_{j}\}_{j=1}^{n}\subset H$ is $\delta$-quasi-orthogonal, there exists $\mathcal{U}=(u_{1}, u_{2}, \ldots, u_{n})$ satisfying that 
	\begin{align}\label{ine:vec_mar}
		\mathcal{U}^{T}\mathcal{U}=I,\quad \Vert \mathcal{U}-\mathcal{V}\Vert\leqslant\sqrt{\sum_{i=1}^{n}\Vert u_{i}-v_{i}\Vert^{2}}\leqslant\sqrt{n}\delta.
	\end{align}
	
	Let $\operatorname{span}(\mathcal{V})=\operatorname{span}\{v_{1},\ldots,v_{n}\}$. If $\mathcal{W}\in H^{n}$ with $\mathcal{W}^{T}\mathcal{W}=I$ and $\operatorname{span}(\mathcal{W})=\operatorname{span}(\mathcal{V})$, then there exists the polar decomposition of $\mathcal{W}^{T}\mathcal{V}$ as follows,
	\begin{align}\label{eq:m_polar_dec}
		\mathcal{W}^{T}\mathcal{V}=QP,
	\end{align}
	where $Q\in\mathbb{R}^{n\times n}$ is an orthogonal matrix and $P\in\mathbb{R}^{n\times n}$ is positive definite. Denote $\tilde{\mathcal{V}}=\mathcal{W}Q$, and we see that $\tilde{\mathcal{V}}^{T}\tilde{\mathcal{V}}=Q^{T}\mathcal{W}^{T}\mathcal{W}Q=I$ and
	\begin{align}\label{eq:f_polar_dec}
		\mathcal{V}=\mathcal{W}\mathcal{W}^{T}\mathcal{V}=\mathcal{W}QP=\tilde{\mathcal{V}}P.
	\end{align}
	The decomposition \cref{eq:f_polar_dec} is unique since \cref{eq:m_polar_dec} implies
	\begin{align*}
		P=\sqrt{(\mathcal{W}^{T}\mathcal{V})^{T}\mathcal{W}^{T}\mathcal{V}}=\sqrt{\mathcal{V}^{T}\mathcal{W}^{T}\mathcal{W}\mathcal{V}}=\sqrt{\mathcal{V}^{T}\mathcal{V}}.
	\end{align*}
	
	Next, we show that the maximum singular value $\sigma_{\text{max}}(\mathcal{U}^{T}\tilde{\mathcal{V}})\leqslant1$. Consider $\mathcal{Z}\in H^{m}(m\geqslant n)$ with $\mathcal{Z}^{T}\mathcal{Z}=I$ and $\operatorname{span}(\mathcal{Z})=\operatorname{span}(\mathcal{U})+\operatorname{span}(\tilde{\mathcal{V}})$. We observe that there exist $Q_{1}, Q_{2}\in\mathbb{R}^{m\times n}$ with $Q_{1}^{T}Q_{1}=Q_{2}^{T}Q_{2}=I$ satisfying
	\begin{align*}
		\mathcal{U}=\mathcal{Z}Q_{1}, \quad \tilde{\mathcal{V}}=\mathcal{Z}Q_{2},
	\end{align*}
	and
	\begin{align}\label{ine:less2}
		\sigma_{\text{max}}(\mathcal{U}^{T}\tilde{\mathcal{V}})=\sigma_{\text{max}}((\mathcal{Z}Q_{1})^{T}\mathcal{Z}Q_{2})=\sigma_{\text{max}}(Q_{1}^{T}Q_{2})\leqslant\sigma_{\text{max}}(Q_{1})\sigma_{\text{max}}(Q_{2})\leqslant1.
	\end{align}
	
	It follows from \cref{eq:f_polar_dec} and \cref{ine:less2} that 
	\begin{align*}
		\Vert \mathcal{U}-\mathcal{V}\Vert^{2}=&\Vert \mathcal{U}\Vert^{2}+\Vert \mathcal{V}\Vert^{2}-2\operatorname{tr}(\mathcal{U}^{T}\mathcal{V})\geqslant\Vert \tilde{\mathcal{V}}\Vert^{2}+\Vert \mathcal{V}\Vert^{2}-2\sum_{i=1}^{n}\sigma_{i}(\mathcal{U}^{T}\mathcal{V})\\=&\Vert \tilde{\mathcal{V}}\Vert^{2}+\Vert \mathcal{V}\Vert^{2}-2\sum_{i=1}^{n}\sigma_{i}(\mathcal{U}^{T}\tilde{\mathcal{V}}\tilde{\mathcal{V}}^{T}\mathcal{V})\\\geqslant&\Vert \tilde{\mathcal{V}}\Vert^{2}+\Vert \mathcal{V}\Vert^{2}-2\sum_{i=1}^{n}\sigma_{\text{max}}(\mathcal{U}^{T}\tilde{\mathcal{V}})\sigma_{i}(P)\\\geqslant&\Vert \tilde{\mathcal{V}}\Vert^{2}+\Vert \mathcal{V}\Vert^{2}-2\operatorname{tr}(\tilde{\mathcal{V}}^{T}\mathcal{V})=\Vert \tilde{\mathcal{V}}-\mathcal{V}\Vert^{2}.
	\end{align*}
	Hence, we have 
	\begin{align*}
		\Vert \tilde{\mathcal{V}}-\mathcal{V}\Vert\leqslant \Vert \mathcal{U}-\mathcal{V}\Vert,
	\end{align*}
	which together with \cref{ine:vec_mar} yields that for $\tilde{\mathcal{V}}=(\tilde{v}_{1}, \tilde{v}_{2}, \ldots, \tilde{v}_{n})$,
	\begin{align*}
		\Vert \tilde{v}_{j}-v_{j}\Vert\leqslant\Vert \tilde{\mathcal{V}}-\mathcal{V}\Vert\leqslant\Vert \mathcal{U}-\mathcal{V}\Vert\leqslant\sqrt{n}\delta,\quad j=1,2,\ldots,n.
	\end{align*}
\end{proof}
\subsection{Proof of \cref{thm:fram_keep_dim}}\label{sec:p_fram_keep_dim}
\begin{proof}
	Let $\{x_{ij}\}_{j=1}^{d_{i}}$ be an orthonormal basis of $X_{i}(i=1,\ldots, p)$ with $a(x_{ij}, x_{kl})=\delta_{ik}\delta_{jl}$. We  obtain from \cref{prop:subspace_angle} that there exists an orthonormal basis $\{y_{ij}\}_{j=1}^{d_{i}}$ of $Y_{i} (i=1,\ldots,q)$ satisfying that for $(1,1)\leqslant(i,j)\leqslant(q,d_{q})$,
	\begin{align}\label{n_intera_keep_dim}
		\operatorname{dist}(x_{ij}, y_{ij})\leqslant \max_{i=1,\ldots,q}(1+\sqrt{d_{i}})\sqrt{2-2\sqrt{1-\varepsilon^{2}}}:=\tilde{\varepsilon}.
	\end{align}
	Obviously, $\varepsilon<\min_{1\leqslant i\leqslant q}\sqrt{\frac{4(1+\sqrt{d_{i}})^{2}N-1}{4(1+\sqrt{d_{i}})^{4}N^{2}}}$ implies that $\tilde{\varepsilon}<\frac{1}{\sqrt{N}}$. 
	Set $\{\beta_{rt}^{(ij)}\}$ such that $ y_{ij}=\sum_{r=1}^{p}\sum_{t=1}^{d_{r}}\beta_{rt}^{(i j)}x_{rt}$ for $ (1,1)\leqslant(i,j)\leqslant(q,d_{q})$ and we have that
	\begin{align*}
		\left(y_{11}, 
		\cdots, y_{1d_{1}}, \cdots, y_{q1}, \cdots, y_{qd_{q}}\right)=\left(x_{11}, \cdots, x_{1d_{1}}, \cdots,  x_{pd_{p}}\right)B_{1},
	\end{align*}
	where
	\begin{align*}
		B_{1}=\begin{pmatrix}\beta_{11}^{(11)}&\cdots&\beta_{11}^{(1d_{1})}&\cdots&\beta_{11}^{(q1)}&\cdots&\beta_{11}^{(qd_{q})}\\\vdots&\ddots&\vdots&\ddots&\vdots&\ddots&\vdots&\\\beta_{1d_{1}}^{(11)}&\cdots&\beta_{1d_{1}}^{(1d_{1})}&\cdots&\beta_{1d_{1}}^{(q1)}&\cdots&\beta_{1d_{1}}^{(qd_{q})}\\\vdots&\ddots&\vdots&\ddots&\vdots&\ddots&\vdots&\\\beta_{p1}^{(11)}&\cdots&\beta_{p1}^{(1d_{1})}&\cdots&\beta_{p1}^{(q1)}&\cdots&\beta_{p1}^{(qd_{q})}\\\vdots&\ddots&\vdots&\ddots&\vdots&\ddots&\vdots&\\\beta_{pd_{p}}^{(11)}&\cdots&\beta_{pd_{p}}^{(1d_{1})}&\cdots&\beta_{pd_{p}}^{(q1)}&\cdots&\beta_{pd_{p}}^{(qd_{q})}\end{pmatrix}:= \begin{pmatrix}B_{2}\\\star\end{pmatrix},
	\end{align*}
	with 
	\begin{align*}
		B_{2}=\begin{pmatrix}\beta_{11}^{(11)}&\cdots&\beta_{11}^{(qd_{q})}\\\vdots&\ddots&\vdots\\\beta_{qd_{q}}^{(11)}&\cdots&\beta_{qd_{q}}^{(qd_{q})}\end{pmatrix}.
	\end{align*}
	
	We assert that $B_{2}$ is strictly diagonally dominant. In fact, we obtain from \cref{n_intera_keep_dim} that 
	\begin{align*}
		\tilde{\varepsilon}\geqslant\operatorname{dist}(x_{ij}, y_{ij})=\operatorname{dist}(y_{ij}, x_{ij})=\frac{\left\Vert\sum_{r=1}^{p}\sum_{t=1}^{d_{r}}\beta_{rt}^{(i j)}x_{rt}-\beta_{ij}^{(i j)}x_{ij}\right\Vert}{\left\Vert\sum_{r=1}^{p}\sum_{t=1}^{d_{r}}\beta_{rt}^{(i j)}x_{rt}\right\Vert}.
	\end{align*}
	Note that
	\begin{align*}
		\left(\beta_{ij}^{(ij)}\right)^{2}\geqslant\left(1-\tilde{\varepsilon}^{2}\right)\sum_{r=1}^{p}\sum_{t=1}^{d_{r}}\left(\beta_{rt}^{(ij)}\right)^{2},\quad j=1,\ldots,d_{i}
	\end{align*}
	implies
	\begin{align*}
		\left|\beta_{ij}^{(ij)}\right|&\geqslant\sqrt{\left(\frac{1}{\tilde{\varepsilon}^{2}}-1\right)\sum_{(r,t)\neq(i,j)}\left(\beta_{rt}^{(ij)}\right)^{2}}>\sqrt{(N-1)\sum_{(r,t)\neq(i,j)}\left(\beta_{rt}^{(ij)}\right)^{2}}\\
		&\geqslant\sum_{(r,t)\neq(i,j)}\left|\beta_{rt}^{(ij)}\right|, \quad (1,1)\leqslant(i,j)\leqslant(q,d_{q}).
	\end{align*}
	It follows from the Gershgorin circle theorem that
	\begin{align*}
		\left|\lambda-\beta_{ij}^{(ij)}\right|\leqslant\sum_{(r, t)\neq(i, j)}\left|\beta_{rt}^{(ij)}\right|,\quad \forall \lambda\in \sigma(B_{2}).
	\end{align*}
	Consequently, we have
	\begin{align*}
		|\lambda|\geqslant\left|\beta_{ij}^{(ij)}\right|-\sum_{(r, t)\neq(i, j)}\left|\beta_{rt}^{(ij)}\right|>0, \quad (1,1)\leqslant(i,j)\leqslant(q,d_{q}),
	\end{align*}
	and 
	$\operatorname{rank}(B_{2})=N$. Therefore, $\operatorname{rank}(B_{1})=N$, which completes the proof.
\end{proof}
\subsection{Proof of \cref{prop:conv_fram}}\label{prof:conv_fram}
\begin{proof}
	Since $\operatorname{dist}(M_{h}(\lambda_{i}), U_{n+1/2}^{(i)})\ll1 (i=1,\ldots,q)$, 
	we obtain from \cref{thm:fram_keep_dim} that 
	\begin{align*}
		U_{n+1}=\bigoplus_{i=1}^{q}U_{n+1/2}^{(i)}. 
	\end{align*}
	
	For $\psi\in \bigoplus_{i=1}^{q}M_{h}(\lambda_{i})$ with $\Vert\psi\Vert=1$ and the orthonormal basis $\{v^{h}_{ij}\}_{j=1}^{d_{i}}$ of $M_{h}(\lambda_{i})$ with $a(v^{h}_{ij}, v^{h}_{kl})=\delta_{ik}\delta_{jl}$, there exists $\{\alpha_{ij}\}$ satisfying $\sum_{i=1}^{q}\sum_{j=1}^{d_{i}}\alpha_{ij}^{2}=1$ and $\psi=\sum_{i=1}^{q}\sum_{j=1}^{d_{i}}\alpha_{ij}v^{h}_{ij}.$  It holds that 
	\begin{equation}\label{ine:piece_cluster}
		\begin{aligned}
			&\operatorname{dist}\left(\psi, U_{n+1}\right)=\left\Vert\left(\operatorname{I}-\mathcal{P}_{U_{n+1}}\right)\sum_{i=1}^{q}\sum_{j=1}^{d_{i}}\alpha_{ij}v^{h}_{ij}\right\Vert\\\leqslant&\sum_{i=1}^{q}\sum_{j=1}^{d_{i}}|\alpha_{ij}|\left\Vert\left(\operatorname{I}-\mathcal{P}_{U_{n+1}}\right)v^{h}_{ij}\right\Vert\leqslant\sqrt{\sum_{i=1}^{q}\sum_{j=1}^{d_{i}}\left\Vert\left(\operatorname{I}-\mathcal{P}_{U_{n+1}}\right)v^{h}_{ij}\right\Vert^{2}}\\=&\sqrt{\sum_{i=1}^{q}\sum_{j=1}^{d_{i}}\operatorname{dist}^{2}\left(v^{h}_{ij}, U_{n+1}\right)}\leqslant\sqrt{\sum_{i=1}^{q}\sum_{j=1}^{d_{i}}\operatorname{dist}^{2}\left(M_{h}(\lambda_{i}), U_{n+1}\right)}\\\leqslant&\sqrt{\sum_{i=1}^{q}\sum_{j=1}^{d_{i}}\operatorname{dist}^{2}\left(M_{h}(\lambda_{i}), U_{n+1/2}^{(i)}\right)}\leqslant\sqrt{N}\max_{1\leqslant i\leqslant q}\operatorname{dist}(M_{h}(\lambda_{i}), U_{n+1/2}^{(i)}),
		\end{aligned} 
	\end{equation}
	and
	\begin{align}\label{ine:piece_to_cluster}
		\operatorname{dist}\left(\bigoplus_{i=1}^{q}M_{h}(\lambda_{i}), U_{n+1}\right)\leqslant\sqrt{N}\max_{1\leqslant i\leqslant q}\operatorname{dist}(M_{h}(\lambda_{i}), U_{n+1/2}^{(i)})\ll1.
	\end{align}
	
	It is clear that $U_{n+1}$ is a finite dimensional subspace of $V^{h}$.  We apply  \cref{thm:cluster_eigen} to \cref{eq:pro_ei_pro} and obtain that $\lambda^{(n+1)}_{ij}\leqslant\lambda_{i+1,1}^{h}$ due to $\operatorname{dist}\left(\bigoplus_{i=1}^{q}M_{h}(\lambda_{i}), U_{n+1}\right)\ll1$ for $(1,1)\leqslant(i, j)\leqslant(q, d_{q})$. 
	Hence, it follows from \cref{ine:piece_to_cluster}, \cref{thm:cluster_eigen} and \cref{prop:clus_to_vec} applied to \cref{eq:pro_ei_pro} that 
	\begin{align*}
		|\lambda_{ij}^{(n+1)}-\lambda^{h}_{ij}|\leqslant&\lambda_{i+1,1}^{h} \operatorname{dist}^{2}\left(\bigoplus_{i=1}^{q}M_{h}(\lambda_{i}), U_{n+1}\right)\leqslant\lambda_{i+1,1}^{h}\sqrt{N}\max_{1\leqslant i\leqslant q}\operatorname{dist}^{2}(M_{h}(\lambda_{i}), U_{n+1/2}^{(i)}),\\
		\operatorname{dist}(u^{h,o}_{ij}, u_{ij}^{(n+1)})\leqslant& \tilde{C}_{**}\operatorname{dist}\left(\bigoplus_{i=1}^{q}M_{h}(\lambda_{i}), U_{n+1}\right)\leqslant \tilde{C}_{**}\sqrt{N}\max_{1\leqslant i\leqslant q}\operatorname{dist}(M_{h}(\lambda_{i}), U_{n+1/2}^{(i)}),
	\end{align*}
	where the constant $\tilde{C}_{**}$ is independent of $U_{n+1/2}$ and may depend on $V^{h}$ (see \cref{thm:cluster_eigenfunc} and \cref{rem:chosen}).
\end{proof}
\subsection{Proof of \cref{thm:multi_keep_dim}}\label{sec:p_multi_keep_dim}
\begin{proof}
	For the $n$-th iteration, consider the linear operators
	\begin{align*}
		\mathcal{F}_{n}^{(i)}:U_{n}^{(i)}\rightarrow U_{n+1/2}^{(i)},\quad i=1,2,\ldots,q,
	\end{align*}
	and for $u^{(n)}\in U_{n}^{(i)}$, $u^{(n+1/2)}=\mathcal{F}_{n}^{(i)}u^{(n)}$ satisfying 
	\begin{align*}
		a\left(u^{(n+1/2)},v\right)-\bar{\lambda}_{i}b\left(u^{(n+1/2)},v\right)=\bar{\lambda}_{i}b\left(u^{(n)},v\right),\quad\forall v\in V^{h}.
	\end{align*}
	
	We claim that $\mathcal{F}_{n}^{(i)}$ is an injection.  Indeed, $u^{(n+1/2)}=v^{(n+1/2)}$ implies that
	\begin{align*}
		\bar{\lambda}_{i}b\left(u^{(n)},v\right)&= a\left(u^{(n+1/2)},v\right)-\bar{\lambda}_{i}b\left(u^{(n+1/2)},v\right)\\
		&=a\left(v^{(n+1/2)},v\right)-\bar{\lambda}_{i}b\left(v^{(n+1/2)},v\right)=\bar{\lambda}_{i}b\left(v^{(n)},v\right),\quad v\in V^{h},
	\end{align*}
	and $u^{(n)}=v^{(n)}$. 
	
	Since $U_{n}^{(i)}$ and $U_{n+1/2}^{(i)}$ are finite dimensional, $\mathcal{F}_{n}^{(i)}$ is indeed an isomorphism and we arrive at
	\begin{align*}
		\operatorname{dim}\left(U_{n+1/2}^{(i)}\right)=\operatorname{dim}\left(U_{n}^{(i)}\right),\quad i=1,2,\ldots,q.
	\end{align*}
\end{proof}
\subsection{Proof of \cref{thm:n_shifted_multi}}\label{sec:p_n_shifted_multi}
 \begin{proof}
	Let us consider $n=0$ first. 
	
	Since $\{u^{h}_{ij}\}_{j=1}^{d_{i}}$ is the orthonormal basis of $M_{h}(\lambda_{i}) (i=1, 2, \ldots, p)$ with $b(u^{h}_{ij}, u^{h}_{kl})=\delta_{ik}\delta_{jl}$, for $v_{i}^{(0)}\in U_{0}^{(i)}$, there exists $\left\{\alpha_{rt}^{(i)}\right\}$ such that 
	\begin{align}\label{eq:par_rep}
		v_{i}^{(0)}=\sum_{r=1}^{p}\sum_{t=1}^{d_{r}}\alpha_{rt}^{(i)}u^{h}_{rt}, \quad i=1, 2, \ldots, q. 
	\end{align}
	
	A simple calculation and the equation
	\begin{align*}
		a(\mathcal{P}_{M_{h}(\lambda_{i})}v_{i}^{(0)}, v)=a(v_{i}^{(0)}, v), \quad v\in M_{h}(\lambda_{i})
	\end{align*}
	show that
	\begin{align*}
		\mathcal{P}_{M_{h}(\lambda_{i})}v_{i}^{(0)}=\sum_{t=1}^{d_{i}}\alpha_{it}^{(i)}u^{h}_{it}.
	\end{align*}
	
	We obtain from \cref{ine:dist_exchange} and \cref{ine:init_theta_multi_keep_dim_1} that there exists $\varepsilon_{0}\in(0,1)$ such that
	\begin{align*}
		\varepsilon_{0}\geqslant&\operatorname{dist}\left(U_{0}^{(i)}, M_{h}(\lambda_{i})\right)\geqslant\operatorname{dist}\left(v_{i}^{(0)}, M_{h}(\lambda_{i}) \right)\\=&\operatorname{dist}\left(v_{i}^{(0)}, \mathcal{P}_{M_{h}(\lambda_{i})}v_{i}^{(0)} \right)=\frac{\left\Vert \sum_{r=1}^{p}\sum_{t=1}^{d_{r}}\alpha_{rt}^{(i)}u^{h}_{rt}-\sum_{t=1}^{d_{i}}\alpha_{it}^{(i)}u^{h}_{it}\right\Vert}{\left\Vert\sum_{r=1}^{p}\sum_{t=1}^{d_{r}}\alpha_{rt}^{(i)}u^{h}_{rt}\right\Vert},
	\end{align*}
	which yields
	\begin{align}\label{ine:main_part_theta_multi}
		\sum_{t=1}^{d_{i}}\lambda^{h}_{it}\left(\alpha_{it}^{(i)}\right)^{2}\geqslant\left(\frac{1-\varepsilon^{2}_{0}}{\varepsilon_{0}^{2}}\right)\sum_{1\leqslant r\neq i\leqslant p}\sum_{t=1}^{d_{r}}\lambda^{h}_{rt}\left(\alpha_{rt}^{(i)}\right)^{2}, \quad i=1, 2, \ldots, q.
	\end{align}
	
	Let $v_{i}^{(1/2)}\in V^{h}$ satisfy
	\begin{align*}
		a\left(v_{i}^{(1/2)},v\right)-\bar{\lambda}_{i}b\left(v_{i}^{(1/2)},v\right)=\bar{\lambda}_{i}b\left(v_{i}^{(0)},v\right)\quad\forall v\in V^{h}.
	\end{align*}
	We may write $v_{i}^{(1/2)}=\sum_{r=1}^{p}\sum_{t=1}^{d_{r}}\beta_{rt}^{(i)}u^{h}_{rt}$ for $i=1, 2, \ldots, q$. Note that \cref{eq:fd_weak_form_leq}, \cref{eq:inv_pow} and  \cref{eq:par_rep} imply
	\begin{align*}
		&\bar{\lambda}_{i}\sum_{r=1}^{p}\sum_{t=1}^{d_{r}}\alpha_{rt}^{(i)}b(u^{h}_{rt}, v)=\bar{\lambda}_{i}b(v_{i}^{(0)}, v)=a\left(v_{i}^{(1/2)},v\right)-\bar{\lambda}_{i}b\left(v_{i}^{(1/2)},v\right)\\=&\sum_{r=1}^{p}\sum_{t=1}^{d_{r}}\left(\lambda^{h}_{rt}-\bar{\lambda}_{i}\right)\beta_{rt}^{(i)}b(u^{h}_{rt}, v),\quad v\in V^{h}.
	\end{align*}
	We have
	\begin{align*}
		\beta_{rt}^{(i)}=\frac{\bar{\lambda}_{i}}{\lambda^{h}_{rt}-\bar{\lambda}_{i}}\alpha_{rt}^{(i)}, \quad i=1, 2, \ldots, q,
	\end{align*}
	and hence, 
	$$v_{i}^{(1/2)}=\sum_{r=1}^{p}\sum_{t=1}^{d_{r}}\frac{\bar{\lambda}_{i}}{\lambda^{h}_{rt}-\bar{\lambda}_{i}}\alpha_{rt}^{(i)}u^{h}_{rt}, \quad  i=1, 2, \ldots, q. $$
	
	Note that \cref{ine:eig_guess_theta_multi} and \cref{ine:main_part_theta_multi} imply that
	\begin{align*}
		\left|\lambda^{h}_{rt}-\bar{\lambda}_{i}\right|\geqslant&\left|\lambda^{h}_{rt}-\lambda^{h}_{i1}\right|-\left|\lambda^{h}_{i1}-\bar{\lambda}_{i}\right|\geqslant g-\delta_{0}, \quad r\neq i,
	\end{align*}
	and
	\begin{align*}
		&\frac{\sum_{1\leqslant r\neq i\leqslant p}\sum_{t=1}^{d_{r}}\left(\frac{\bar{\lambda}_{i}}{\lambda^{h}_{rt}-\bar{\lambda}_{i}}\alpha_{rt}^{(i)}\right)^{2}\lambda^{h}_{rt}}{\sum_{t=1}^{d_{i}}\left(\frac{\bar{\lambda}_{i}}{\lambda^{h}_{it}-\bar{\lambda}_{i}}\alpha_{it}^{(i)}\right)^{2}\lambda^{h}_{it}}\\\leqslant&\left(\frac{\delta_{0}}{g-\delta_{0}}\right)^{2}\frac{\sum_{1\leqslant r\neq i\leqslant p}\sum_{t=1}^{d_{r}}\lambda^{h}_{rt}\left(\alpha_{rt}^{(i)}\right)^{2}}{\sum_{t=1}^{d_{i}}\lambda^{h}_{it}\left(\alpha_{it}^{(i)}\right)^{2}}\leqslant\left(\frac{\delta_{0}}{g-\delta_{0}}\right)^{2}\frac{\varepsilon_{0}^{2}}{1-\varepsilon_{0}^{2}}.
	\end{align*}
	
	We then get that 
	\begin{equation}
		\begin{aligned}\label{eq:new_ite_theta_multi}
			\operatorname{dist}\left(v_{i}^{(1/2)}, M_{h}(\lambda_{i}) \right)
			=&\frac{\left\Vert v_{i}^{(1/2)}-\mathcal{P}_{M_{h}(\lambda_{i})}v_{i}^{(1/2)}\right\Vert}{\left\Vert v_{i}^{(1/2)}\right\Vert}\\=&\sqrt{1-\frac{ \sum_{t=1}^{d_{i}}\left(\frac{\bar{\lambda}_{i}}{\lambda^{h}_{it}-\bar{\lambda}_{i}}\alpha_{it}^{(i)}\right)^{2}\lambda^{h}_{it}}{\sum_{r=1}^{p}\sum_{t=1}^{d_{r}}\left(\frac{\bar{\lambda}_{i}}{\lambda^{h}_{rt}-\bar{\lambda}_{i}}\alpha_{rt}^{(i)}\right)^{2}\lambda^{h}_{rt}}}\\\leqslant&\frac{\delta_{0}\varepsilon_{0}}{\sqrt{(g-\delta_{0})^{2}(1-\varepsilon_{0}^{2})+\delta_{0}^{2}\varepsilon^{2}_{0}}}:=\varepsilon_{1},\quad i=1,\ldots,q.
		\end{aligned}
	\end{equation}
	
	We see that $v_{i}^{(1)}=\mathcal{F}_{0}^{(i)}v_{i}^{(0)}$, where $\mathcal{F}_{0}^{(i)}$ is an isomorphism. Then we obtain from \cref{ine:eig_guess_theta_multi} and \cref{eq:new_ite_theta_multi} that $\varepsilon_{1}\leqslant\varepsilon_{0}<1$ and 
	\begin{align*}
		\operatorname{dist}\left(M_{h}(\lambda_{i}),  U_{1/2}^{(i)} \right)\leqslant\varepsilon_{1},\quad i=1,2,\ldots,q.
	\end{align*}
	
	Similarly, we have
	\begin{align*}
		\operatorname{dist}\left(M_{h}(\lambda_{i}), U_{n+1/2}^{(i)} \right)\leqslant\varepsilon_{n+1}=\frac{\delta_{0}\varepsilon_{n}}{\sqrt{(g-\delta_{0})^{2}(1-\varepsilon_{n}^{2})+\delta_{0}^{2}\varepsilon^{2}_{n}}},\quad \forall n\in\mathbb{N}.
	\end{align*}
	
	Thus, $\varepsilon_{n+1}\leqslant\varepsilon_{n},\forall n\in\mathbb{N}$ and 
	\begin{align*}
		\varepsilon_{n+1}=&\frac{\delta_{0}\varepsilon_{n}}{\sqrt{(g-\delta_{0})^{2}(1-\varepsilon_{n}^{2})+\delta_{0}^{2}\varepsilon^{2}_{n}}}\leqslant\left(\frac{\delta_{0}}{\sqrt{(g-\delta_{0})^{2}(1-\varepsilon_{n-1}^{2})+\delta_{0}^{2}\varepsilon^{2}_{n-1}}}\right)^{2}\varepsilon_{n-1}\\\leqslant&\cdots\leqslant\left(\frac{\delta_{0}}{\sqrt{(g-\delta_{0})^{2}(1-\varepsilon_{0}^{2})+\delta_{0}^{2}\varepsilon^{2}_{0}}}\right)^{n}\varepsilon_{0},
	\end{align*}
	which indicates that $\varepsilon_{n+1}$ decreases towards $0$ as $n\rightarrow\infty$. Moreover, there holds that
	\begin{align}\label{lim:depen}
		\lim_{n\rightarrow\infty}\frac{\varepsilon_{n+1}}{\varepsilon_{n}}=\frac{\delta_{0}}{g-\delta_{0}}.
	\end{align}
\end{proof}

\subsection{Proof of \cref{prop:clus_pou_n_shifted}}\label{sec:p_clus_pou_n_shifted}
 \begin{proof}Let us start by $n=0$.
	
	Since $\{u^{h,o,0}_{ij}\}_{i=1}^{d_{i}}$ is an orthonormal basis of $M_{h}(\lambda_{i})$, for $\varphi\in M_{h}(\lambda_{i})$ with $\Vert\varphi\Vert=1$, there exists $\{\alpha_j\}$ satisfying $\varphi=\sum_{j=1}^{d_{i}}\alpha_{j}u^{h,o,0}_{ij}$ and $\sum_{j=1}^{d_{i}}\alpha_{j}^{2}=1$. Note that
	\begin{align*}
		&\operatorname{dist}\left(\varphi, U_{0}^{(i)}\right)=\left\Vert\left(\operatorname{I}-\mathcal{P}_{U_{0}^{(i)}}\right)\varphi\right\Vert\leqslant\sum_{j=1}^{d_{i}}|\alpha_{j}|\left\Vert\left(\operatorname{I}-\mathcal{P}_{U_{0}^{(i)}}\right)u^{h,o,0}_{ij}\right\Vert\\=&\sum_{j=1}^{d_{i}}|\alpha_{j}|\operatorname{dist}\left(u^{h,o,0}_{ij}, U_{0}^{(i)}\right)\leqslant\sum_{j=1}^{d_{i}}|\alpha_{j}|\operatorname{dist}\left(u^{h,o,0}_{ij}, u_{ij}^{(0)}\right)\leqslant\sqrt{d_{i}}\varepsilon_{0}.
	\end{align*}
	We arrive at
	\begin{align}\label{ine:clu_pieces}
		\operatorname{dist}\left(M_{h}(\lambda_{i}), U_{0}^{(i)}\right)\leqslant\sqrt{d_{i}}\varepsilon_{0}\leqslant\sqrt{D}\varepsilon_{0},\quad i=1,\ldots,q.
	\end{align}
	
	For $(1,1)\leqslant(i,j)\leqslant(q,d_{q})$, we have from \cref{eq:eigensagf} that
	\begin{align*}
		\left|\lambda^{h}_{ij}-\bar{\lambda}_{i}\right|\leqslant\left|\lambda^{h}_{ij}-\mathcal{C}_{i}\left(\left\{\lambda^{h}_{ij}\right\}_{j=1}^{d_{i}}\right)\right|+\left|\mathcal{C}_{i}\left(\left\{\lambda^{h}_{ij}\right\}_{j=1}^{d_{i}}\right)-\bar{\lambda}_{i}\right|\leqslant\gamma+\zeta_{0}<\frac{g}{2}.
	\end{align*}
	Consider $U_{1/2}^{(i)}=\operatorname{span}\{u_{i1}^{(1/2)}, \ldots, u_{id_{i}}^{(1/2)}\}$ for $i=1,\ldots,q$. In accordance with \cref{thm:n_shifted_multi} and \cref{ine:clu_pieces}, there holds that
	\begin{align*}
		\operatorname{dist}\left(M_{h}(\lambda_{i}), U_{1/2}^{(i)}\right)\leqslant\frac{\sqrt{D}\left(\gamma+\zeta_{0}\right)\varepsilon_{0}}{\sqrt{(g-\gamma-\zeta_{0})^{2}(1-D\varepsilon_{0}^{2})+D\left(\gamma+\zeta_{0}\right)^{2}\varepsilon^{2}_{0}}},\quad i=1,\ldots,q.
	\end{align*}
	
	Since $\varepsilon_{0}$ is sufficiently small, we obtain from \cref{thm:fram_keep_dim} that $U_{1}=\bigoplus_{i=1}^{q}U_{1/2}^{(i)}$, which together with \cref{thm:multi_keep_dim} implies that
	\begin{align*}
		\operatorname{dim}(U_{1})=\sum_{i=1}^{q}\operatorname{dim}(U_{1/2}^{(i)})=\sum_{i=1}^{q}d_{i}=N.
	\end{align*}
	
	Due to \cref{ine:piece_to_cluster}, we have
	\begin{align*}
		\operatorname{dist}\left(\bigoplus_{i=1}^{q}M_{h}(\lambda_{i}), U_{1}\right)\leqslant\frac{\sqrt{DN}\left(\gamma+\zeta_{0}\right)\varepsilon_{0}}{\sqrt{(g-\gamma-\zeta_{0})^{2}(1-D\varepsilon_{0}^{2})+D\left(\gamma+\zeta_{0}\right)^{2}\varepsilon^{2}_{0}}}:=\xi_{1}.
	\end{align*}
	
	We apply \cref{prop:conv_fram} to \cref{eq:pou_n_shifted_sub} and obtain that there exists an orthonormal basis  $\{u^{h,o,1}_{ij}\}_{j=1}^{d_{i}}$ of $M_{h}(\lambda_{i})$  with $b(u^{h,o,1}_{ij}, u^{h,o,1}_{kl})=\delta_{ik}\delta_{jl}$  such that 
	\begin{align}\label{resu:the_first}
		\operatorname{dist}\left(u^{h,o,1}_{ij}, u_{ij}^{(1)}\right)\leqslant \tilde{C}_{**}\xi_{1}, \quad\lambda_{ij}^{(1)}-\lambda^{h}_{ij}\leqslant\lambda^{h}_{q+1,1}\xi_{1}^{2},\quad (1,1)\leqslant(i,j)\leqslant(q,d_{q}),
	\end{align}
	where $\tilde{C}_{**}$ is a constant that is independent of $U_{1}$, i.e., independent of the iteration. 
	
	Set $\varepsilon_{1}=\tilde{C}_{**}\xi_{1}$ and $\zeta_{1}=\lambda^{h}_{q+1,1}\xi_{1}^{2}$, we then have
	\begin{align*}
		\left|\mathcal{C}_{i}\left(\left\{\lambda^{h}_{ij}\right\}_{j=1}^{d_{i}}\right)-\bar{\lambda}_{i}^{(1)}\right|=\left|\mathcal{C}_{i}\left(\left\{\lambda^{h}_{ij}-\lambda_{ij}^{(1)}\right\}_{j=1}^{d_{i}}\right)\right|\leqslant\zeta_{1}.
	\end{align*}
	
	We obtain that $\varepsilon_{1}\leqslant\varepsilon_{0}$ and $\zeta_{1}\leqslant\zeta_{0}$ since $\varepsilon_{0}, \gamma$ and $\zeta_{0}$ are sufficiently small. Similarly, there hold that
	\begin{align*}
		\operatorname{dist}\left(\bigoplus_{i=1}^{q}M_{h}(\lambda_{i}), U_{n+1}\right)\leqslant&\xi_{n+1}=\frac{\sqrt{DN}\left(\gamma+\zeta_{n}\right)\varepsilon_{n}}{\sqrt{(g-\gamma-\zeta_{n})^{2}(1-D\varepsilon_{n}^{2})+D\left(\gamma+\zeta_{n}\right)^{2}\varepsilon^{2}_{n}}},\\
		0\leqslant\lambda_{ij}^{(n+1)}-\lambda^{h}_{ij}\leqslant&\zeta_{n+1}=\lambda^{h}_{q+1,1}\xi_{n+1}^{2},\quad (1,1)\leqslant(i,j)\leqslant(q,d_{q}).
	\end{align*}
	
	Therefore there exists an orthonormal basis $\{u^{h,o,n+1}_{ij}\}_{j=1}^{d_{i}}$ of $M_{h}(\lambda_{i})$ such that 
	\begin{align*}
		\operatorname{dist}\left(u^{h, o,n+1}_{ij}, u_{ij}^{(n+1)}\right)\leqslant\varepsilon_{n+1}=& \tilde{C}_{**}\xi_{n+1}=\frac{\tilde{C}_{**}\sqrt{DN}\left(\gamma+\zeta_{n}\right)\varepsilon_{n}}{\sqrt{(g-\gamma-\zeta_{n})^{2}(1-D\varepsilon_{n}^{2})+D\left(\gamma+\zeta_{n}\right)^{2}\varepsilon^{2}_{n}}},
	\end{align*}
	for $(1,1)\leqslant(i,j)\leqslant(q,d_{q})$, where the constant $\tilde{C}_{**}$ is the same as the one in \cref{resu:the_first} due to the independence of iterations.
	
	We see that $\varepsilon_{n+1}\leqslant\varepsilon_{n}$ and $\zeta_{n+1}\leqslant\zeta_{n}$, and both $\varepsilon_n$ and $\zeta_{n}$ decrease towards $0$ as $n\rightarrow\infty$. Finally, we arrive at
	\begin{align*}
		\lim_{n\rightarrow\infty}\frac{\varepsilon_{n+1}}{\varepsilon_{n}}=\frac{\tilde{C}_{**}\sqrt{DN}\gamma}{g-\gamma}.
	\end{align*}
\end{proof}

\bibliographystyle{siamplain}
\bibliography{references}
\end{document}